\documentclass[11pt,reqno]{amsart}

\usepackage{amsmath}
\usepackage{amsthm}
\usepackage{graphicx}
\usepackage{mathtools}
\usepackage{amsfonts}
\usepackage{amssymb}
\usepackage{hyperref}
\usepackage{bm}
\usepackage[margin=1 in]{geometry}
\usepackage{enumerate}
\usepackage[normalem]{ulem}
\usepackage{tikz}
\usepackage{tikz-cd}
\usepackage{xcolor}
\usetikzlibrary{matrix,arrows,positioning,automata}
  \usepackage{cancel}
\usepackage[numbers]{natbib}
\usepackage{todonotes}
\usepackage{ulem}

\numberwithin{equation}{section}

\newtheorem{theorem}{Theorem}[section]
\newtheorem{lemma}[theorem]{Lemma}
\newtheorem{proposition}[theorem]{Proposition}

\newtheorem{assumption}[theorem]{Assumption}

\theoremstyle{definition}
\newtheorem{example}[theorem]{Example}\newtheorem{definition}[theorem]{Definition}
\newtheorem{remark}[theorem]{Remark}

\def\E{{\mathbb E}}

\def\R{{\mathbb R}}
\def\N{{\mathbb N}}

\def\P{{\mathcal P}}

\def\X{{\mathcal X}}

\def\L{{\mathcal L}}

\def\F{{\mathcal F}}

\newcommand{\bx}{\mathbf{x}}

\newcommand{\by}{\mathbf{y}}

\DeclareMathOperator*{\argmin}{arg\,min}

\newcommand\cC{\mathcal C}

\newcommand\cF{\mathcal F}

\newcommand\cL{\mathcal L}

\newcommand\cP{\mathcal P}

\newcommand\cW{\mathcal W}
\newcommand\cX{\mathcal X}

\def \E{\mathbb{E}}
\def \F{\mathbb{F}}

\def \L{\mathbb{L}}

\def \N{\mathbb{N}}

\def \P{\mathbb{P}}

\def \R{\mathbb{R}}

\def \X{\mathbb{X}}

\def \e{\mathrm{e}}

  \thanks{The authors wish to thank Vincent Hoffmann and Kaiwen Zhang for helpful comments. The first author is partially supported by the NSF under Grant CAREER DMS-2143861. The second author is partially supported by NSF grant \#DMS-2508581. }

\title[Particle systems and large static games]
{Particle system approximation of Nash equilibria in large games}
\author{Ludovic Tangpi and Nizar Touzi} 
\date{\today}
  \address{Princeton University}
  \email{ludovic.tangpi@princeton.edu}
  \address{New York University}
  \email{nizar.touzi@nyu.edu}
\begin{document}

\begin{abstract}
We develop a probabilistic framework to approximate Nash equilibria in symmetric $N$-player games in the large population regime, via the analysis of associated mean field games (MFGs). The approximation is achieved through the analysis of a McKean–Vlasov type Langevin dynamics and their associated particle systems, with convergence to the MFG solution established in the limit of vanishing temperature parameter. Relying on displacement monotonicity or Lasry–Lions monotonicity of the cost function, we prove contractivity of the McKean–Vlasov process and uniform-in-time propagation of chaos for the particle system. Our results contribute to the general theory of interacting diffusions by showing that monotonicity can ensure convergence without requiring small interaction assumptions or functional inequalities.
\end{abstract}

\maketitle



\section{Introduction}

A classical optimization problem is the analysis of non-cooperative games among $N$ agents with cost functions $F_i:(\R^d)^N\to\R$, $i=1,\dots, N$ for two given integers $d,N$.
For such games, it is standard to search for \emph{Nash equilibria}.
That is, strategies $\bx = (x^1,\dots, x^N)\in (\R^d)^N$ such that
\begin{equation*}
  F_i(\bx) \le F_i(y,\bx^{-i})\quad \text{for all } i\ge 1\text{ and } y\in \R^d,
\end{equation*}
where we put $(\bx^{-i},y) := (x^1,\dots,x^{i-1},y,x^{i+1},\dots,x^N)$.
A famous result by \citet{nash1951noncooperative,Nash50} gives general conditions guaranteeing existence of (relaxed, or mixed) Nash equilibria.
This concept of equilibrium is widely used in several engineering and social sciences applications.
Unfortunately, existence proofs of Nash equilibria are non-constructive. The efficient numerical approximation of Nash equilibria is a challenging problem extensively addressed in the existing literature, see e.g. \citet{NinaBalcan17}, \citet{Borgens21}, \citet{L.Combettes15}, \citet{DGP06}, \citet{Dask09}, \citet{Hansen74}, \citet{Pot-Nud-Sho08}, \citet{Mazumdar-Rat-Sas20}, \citet{Yi-Pavel19} and \citet{Flam93}.
This literature shows in particular that Nash equilibria cannot be computed in polynomial times in general \cite{Dask26,DGP06,Dask09}.

A standard approach to the approximation problem is the classical gradient descent algorithm
\begin{equation*}
  x^i_{t_{k+1}} = x^i_{t_{k}} - D_iF_i(x^1_{t_k},\dots, x^N_{t_k})\delta,
  ~~i=1,\ldots,N
\end{equation*}
for some $\delta>0$ and with given initialization $\bx_{t_0} \in (\R^d)^N$.
The convergence of this algorithm is studied for instance in \citet{Cher-Gha-Cor17}, \citet{Heusel-Ram-Unte17}, \citet{Nagaranjan17} and \citet{Jin-Net-Jor20} in the case of two-player zero-sum game and training of generative adversarial networks and \citet{Mazumdar-Rat-Sas20} in the case of multi-agent games.
In particular, these authors show that unless strong convexity conditions are imposed, the above gradient based algorithm may not converge, or could converge to a non-Nash equilibrium point even in the two-player zero sum case.
When convergence to an equilibrium occurs, the speed of convergence deteriorates when the population size $N$ becomes large.

In order to speed up the convergence of the method and overcome possible lack of convexity of the cost function, a natural idea is to penalize the above scheme by $ \sqrt{2\sigma}\varepsilon^i$ for some Gaussian random variables $\varepsilon^i$ and temperature parameter $\sigma>0$.
We refer for instance to \citet{Chiang-Hwang-Sheu87,Geman-Hwang86} for early works developing the idea. 
The high level idea is that adding the noise $\sqrt{2\sigma}\varepsilon$ allows to gain some convexity of sorts; see e.g. \citet{Huang-Malik25}.
Despite the benefits of the noise in speeding-up the computation and avoiding non-Nash equilibria, the scheme remains slow for large $N$.
Inspired by the theory of mean field games of \citet{lasry2006jeux,lasry2006jeux2} and \citet{huang2006large}, we will focus on symmetric games when the cost functions $F_i$ are given by
\begin{equation*}
  F_i(x^1,\dots,x^N) = F(x_i, \mu^N_{\bx}),~\mbox{with}~\mu^N_\bx := \frac1N\sum_{i=1}^N\delta_{x^i}
  ~\mbox{for all}~\bx = (x^1,\dots,x^N).
\end{equation*}
Here the function $F:\R^d\times \cP(\R^d)\to \R^d$ where $\cP(\R^d)$ is the set of probability measures on $\R^d$.
The associated \emph{(static) mean field game} equilibrium is a probability measure $m\in \cP(\R^d)$ such that 
\begin{equation*}
  \mathrm{supp}(m) = \argmin_{x\in \R^d} F(x,m),
\end{equation*}
where $\mathrm{supp}(m)$ denotes the support of $m$.
The main goal of the present paper is to propose an approximation method for the mean field equilibrium.
This will therefore allow to approximate $N$-player Nash equilibria for $N$ large enough.
Incidentally, our analysis will also allow to show quantitative convergence of Nash equilibria to mean field equilibrium as known for dynamic stochastic differential games, see e.g. \citet{cardaliaguet2017convergence}, \citet{laurieretangpi}, \citet{jackson2023quantitative}, and \citet{gangbo2022mean}.

To achieve these goals we will rely on the interacting particle system
\begin{equation}
  \label{eq:N.SymLangevin}
  dX^i_t = - \nabla_{\!x}F(X^i_t,\mu^N_{\X_t}) dt + \sqrt{2\sigma} dW^i_t,  \mbox{ with } \X = (X^1,\dots, X^N),
\end{equation}
whose Euler-Maruyama discretization is exactly the scheme introduced above.
Its natural $N\to\infty$ limit is the McKean-Vlasov Langevin dynamics
\begin{equation}
\label{eq:McVLangevin}
  dX_t = - \nabla_xF(X_t,m_{X_t}) dt + \sqrt{2\sigma} dW_t \quad\text{with } m_{X_t} := \P\circ X^{-1}_t,
\end{equation} 
for some Brownian motion $W$.

\subsection{Description of the main results}
Similar to Markov chain Monte Carlo methods popular in optimization, the main results of this work will allow to estimate the mean field equilibrium via solutions of \eqref{eq:McVLangevin} (resp. \eqref{eq:N.SymLangevin}) for $t$ large enough and $\sigma$ small enough (resp. for $t,N$ large enough and $\sigma$ small enough).
The main structural properties we will consider on the cost function $F$ are \emph{monotonicity} properties.
We will assume that $F$ is either \emph{Lasry-Lions monotone}, or \emph{displacement monotone}, see Definition \ref{ass.convex}, two alternative conditions whose strict versions guarantee uniqueness of the mean field equilibrium, but also of the invariant measure of \eqref{eq:McVLangevin}. 
The main results of the paper can therefore be summarized as follows:
\begin{enumerate}
    \item[(i)]
    In Theorem \ref{thm:Laplace.principle}, we construct mean field equilibrium from the particle system \eqref{eq:McVLangevin} as $\sigma\to0$ limits of invariant measures.
    \item[(ii)] 
    If $F$ is $\ell_F$-displacement monotone, we show that if $\ell_F >0$, then \eqref{eq:McVLangevin} satisfies contractivity and \eqref{eq:N.SymLangevin} satisfies uniform in time propagation of chaos.
    \item[(iv)]
    Under the above assumptions, we derive convergence of the laws of $X^{i,N}_t$ and $X_t$ to the mean field equilibrium, as well as quantitative convergence (along with concentration inequalities) of the empirical measure $\mu_\bx^N$ to the mean field equilibrium where $\bx$ is the $N$-player Nash equilibrium.
\end{enumerate}

\subsection{Related Literature}
The derivation of explicit convergence rates to the invariant measures of SDEs is a basic and well-studied question in applied probability and PDE.
Due to modern machine leaning applications such as the training of Generative Adversarial Networks and performance of large neural network models, the case of McKean-Vlasov equations has received sustained interest in the recent years, see \citet{Hu-Ren-Sis-Sz21}, \citet{Con-Kaz-Ren22}, \citet{Huang-Zhang23}, \citet{Dus-Ehrl25}, \citet{Maass-Fontbona24} and \citet{Shevchenko22}.
Uniform-in-time propagation of chaos as well as contractivity of McKean-Vlasov equations are studied by many authors.
Most works focus on equations with drift of the form $ - \int \nabla V(x- y)\mu(dy) -\nabla U(x)$ for some external/confinement potential $U$ and an interaction kernel $V$ and they typically fall under one of the following three categories:
Earlier papers assumed strong convexity conditions on the parameters: \citet{Car-McC-Vil03,Malrieu01}, \citet{Monmarche17}.
The convexity assumptions have been substantially weaken to allow weak convexity, but assuming small enough interaction: \citet{Eberle-Guillin20}, \citet{Eberle-Guillin-Zimmer19}, \citet{Lacker-LeFlem03} and  \citet{Schuh24}. The third strand of works rather relies on log-Sobolev inequalies for the invariant measure (or uniform in $N$ log-Sobolev equality for the particle system): \citet{Chizat22}, \citet{Gui-Mon21}, \citet{Chen-Lin-Ren-Wang24}. We notice that, except for \cite{Chen-Lin-Ren-Wang24,Hu-Ren-Sis-Sz21}, most of the existing literature focuses on problems with linear interaction. 
More importantly, essentially all these papers make a crucial smallness assumption on the Lipschitz constant of the interaction term.
In the current work, we observe that monotonicity properties in the measure argument (along with convexity of the function $F$) allow to obtain exponential contractivity of the McKean-Vlasov equation \eqref{eq:McVLangevin} without limitation on the Lipschitz-constant of $F$ or a-priori assuming functional inequalities.

While diffusion based methods for optimization and two-player zero-sum games is an extensively studied topic, we do not know of works on non-zero-sum games.
The closest papers to ours are \citet{Awi-Hynd-Mawi23,RHynd23}, on general $N$-player games and a section on mean field games in \citet{RHynd23}.
Here the authors use PDE techniques to approximate relaxed equilibrium as Ces\`aro limits of dynamical systems.
In these paper, convergence rates are not derived, but similar to the present work, monotonicity properties of the cost functions play a key role for the success of the arguments.

When the function $F$ is derived from a potential $\cF:\cP(\R^d)\to \R$, the game is said to be \emph{potential}, and our method allows to approximate the minimizer of $\cF$ via mean field Langevin dynamics.
Contractivity and propagation of chaos have been obtained notably 
by \citet{Hu-Ren-Sis-Sz21},  \citet{Eberle-Guillin20,Eberle-Guillin-Zimmer19}, mostly focusing on the case
\begin{equation*}
  \cF(m) = \int_{\R^d}V(x)m(dx).
\end{equation*}
In this setting, recent works have made progress allowing to relax the classical smallness assumption on the Lipschitz constant of the interaction by exploiting monotonicity conditions, in the spirit of those used here.
In fact, assuming convexity of $\mathcal{F}$ (and additional a priori functional inequalities) one can show that the Gibbs measure $m^N(d\bx) = \exp(-N\cF(\mu^N_{\bx}))d\bx/Z$ (for some constant $Z$) satisfies a uniform in $N$ log-Sobolev inequality, see \citet{SongboWang24}.
It is well known that such a property allows to derive a uniform in $N$ contractivity result for the $N$-particle system \citet[Section 5.2]{Bak-Gen-Led2014}.
We also refer to \citet{Kook.etal24,Monmarche24} for further development along these lines.

\subsection{Outline of the rest of the paper}

In the next section we present the probabilistic setting and rigorously define the $N$-player game and the associated mean field game.
In the next section we present the main results of the paper, and Sections \ref{sec:mean_field_games_and_ergodicity_of_langevin_dynamics}, \ref{sec:particle_system_contractivity_and_propagation_of_chaos} and \ref{sec:approximation_of_equilibria} is dedicated to the proofs of the three main results discussed above.

\section{The general framework}

\subsection{Notations}

Throughout, we equip $\R^d$ with the Euclidean norm $|\cdot|$, for all $d\in \N^*$. We denote by $\cP_p(\R^d)$ the collection of all probability measures $m$ on $\R^d$ with finite $p-$th moment $|m|_p^p:=\int|x|^pm(dx)<\infty$. Unless otherwise specified, $\cP_p(\R^d)$ is equipped with the $p-$Wasserstein distance:
$$
\cW_p^p(m,m')
:=
\inf_{\pi\in\Pi(m,m')} \int |x - x'|^p\pi(dx,dx'),
~~
m,m'\in\cP_p(\R^d),
$$
with $\Pi(m,m')$ the collection of all coupling probability measures of $m$ amd $m'$, i.e. all probability measures on $\R^d\times\R^d$ with marginals $m$ and $m'$. 

For a vector $\bx:=(x^1,\dots,x^N)\in (\R^d)^{N}$, we denote by
$$
\mu^N_{\bx}:=\frac1N\sum_{i=1}^N \delta_{x^i} \in\cP_p(\R^d)
$$
the corresponding empirical measure, and
$$
\bx^{-i}:=(x^1,\dots,x^{i-1},x^{i+1},\dots,x^N)\in (\R^d)^{N-1}
,~
(\bx^{-i},y):= (x^1,\dots,x^{i-1},y,x^{i+1},\dots,x^N)\in (\R^d)^{N}.
$$

A map $f:\cP_p(\R^d)\longrightarrow\R$ is said to have a linear functional derivative if there exists a continuous function $\delta_mf:\cP_p(\R^d) \times \R^d \longrightarrow\R$, with $p-$polynomial growth in $x$, locally uniformly in $m$, such that for all $m_0,m_1\in\cP_p(\R^d)$ with barycenter $m_\lambda:=(1-\lambda)m_0+\lambda m_1$, we have
$$
\frac{f\big(m_\lambda\big)-f(m_0)}{\lambda}
\longrightarrow
\int_{\R^d} \delta_mf(m_0,x)(m_1-m_0)(dx),
~\mbox{as}~\lambda\searrow 0.
$$
Clearly, $\delta_mf$ is unique up to an additive constant. 
Whenever $\delta_mf(m,.)$ is differentiable, its gradient
\begin{equation*}
D_mf(m,x): = \nabla_{\!x}\delta_mf(m,x)
~~\mbox{for all}~~
m\in\cP_p(\R^d),~x\in\R^d,
\end{equation*}
essentially coincides with the Wasserstein gradient and the Lions' derivative (under additional technical conditions).
Finally, a map $V:\R^d\times \cP_2(\R^d)\to \R$ is said to be continuously differentiable (i.e.~is $C^1$) if for every $m\in \cP_2(\R^d)$, the function $V(\cdot,m)$ is continuously differentiable with jointly continuous partial gradient $\nabla_{\!x}V$, and for every $x\in \R^d$, the function $V(x,\cdot)$ admits a Lions derivative. We refer the reader to \citeauthor{carmona2018probabilisticI} \cite[Chapter 5]{carmona2018probabilisticI} or \citeauthor{cardaliaguet2019master} \cite[Chapter 1]{cardaliaguet2019master} for details.

Throughout this paper, we shall consider a family $(W^1,\dots,W^N)$ of independent Brownian motions with values in $\R^d$, on a filtered probability space $(\Omega^N, \cF^N, \P^N, \F^N)$.
For simplicity, we will assume that $\Omega^N:= (\R^d)^N\times \cC^N$ where $\cC=C([0,T],\times \R^d)$ is the set of continuous functions on $[0,T]$ with values in $\R^d$, $\P^N:=(m_0)^{\otimes N}\times (\P^0)^{\otimes N}$ where $\P^0$ is the Wiener measure on $\cC$ and $m_0$ is some fixed probability measure on $\R^d$ with finite second moment and $\cF^N$ is the Borel-sigma algebra on $\Omega^N$.
We will always put
\begin{equation*}
    (\Omega, \cF, \P) := (\Omega^1,\cF^1,\P^1),\quad W := W^1.
\end{equation*}

\subsection{Finite population game}

The game is defined by $N$ loss functions $F_i:(\R^d)^N\longrightarrow \R$. In this paper, we are concerned with the following notions of equilibria. 

\begin{definition}
{\it
\noindent {\rm (i)} For $\varepsilon\ge 0$, a strategy $\bx\in (\R^d)^N$ is an $\varepsilon-$approximate Nash equilibrium if
\begin{equation*}
    F_i(\bx) \le F_i( \bx^{-i},y) + \varepsilon,
    ~\text{for all}~i=1,\dots,N,~\mbox{and}~ y\in \R^d.
\end{equation*}
{\rm (ii)} A Nash equilibrium is a $0-$approximate Nash equilibrium.}
\end{definition}

Following the traditional gradient descent technique in optimization theory, one may approximate Nash equilibria by means of the solution of the ordinary differential equation defined as the continuous time limit of the agent-wise gradient descent algorithm:
\begin{equation}
\label{eq:sigma=0.equation}
    du^{i,N}_t = - \nabla_{\!x^i}F_i(u^{1,N}_t,\dots, u^{N,N}_t) dt, \quad u^{i,N}_0\in \R^d,~~i=1,\ldots,N.
\end{equation}
To guarantee that the ODE \eqref{eq:sigma=0.equation} has a unique solution $\mathbf{u}=(u^{1,N},\ldots,u^{N,N})$ for all initial condition $\mathbf{u}_0\in (\R^d)^N$, we assume that 
\begin{equation}\label{nablaFLip}
F_i
~\mbox{is continuously differentiable in $x^i$ with Lipschitz-continuous}
~\nabla_{\!x_i}F_i,
~i=1,\ldots,N.
\end{equation}

\begin{definition}\label{def:nonsym}
The family of cost functions $\{F_i\}_{1\le i\le N}$ is strictly displacement monotone if 
  \begin{equation}
  \label{eq:GenDM}
    \sum_{i=1}^N(x^i - y^i)\!\cdot\!\big[\nabla_{\!x^i}F_i(\bx)-\nabla_{\!x^i}F_i(\by)\big]>0,
    ~\mbox{for all}~\bx,\by \in (\R^d)^N~\mbox{with}~\bx \neq \by
\end{equation}
and $\ell_F$-displacement monotone if 
\begin{equation*}
  \sum_{i=1}^N(x^i - y^i)\!\cdot\!\big[\nabla_{\!x^i}F_i(\bx)-\nabla_{\!x^i}F_i(\by)\big]\ge \ell_F\sum_{i=1}^N|x^i - y^i|^2,
\end{equation*}
\end{definition}
We now recall a key result of \citeauthor{Awi-Hynd-Mawi23} \cite{Awi-Hynd-Mawi23} that will be used to show convergence of Nash equilibria to mean field equilibria. First, under the displacement monotonicity condition defined above, if the function
\begin{equation}
\label{eq:convex.group}
  \by \mapsto \sum_{j=1}^NF_j(y^j,x^{-j})\mbox{ is convex for all } \bx,
\end{equation}
and additional regularity properties, the differential equation \eqref{eq:sigma=0.equation} admits a unique Lipschitz continuous solution $\boldsymbol u = (u^{1,N}, \dots, u^{N,N}):[0,\infty)\longrightarrow (\R^d)^N$. More importantly they show that, if the $N$-player game admits a Nash equilibrium, then the Ces\'aro average converges to a Nash equilibrium, i.e.
\begin{equation*}
  \bx = \lim_{t\to \infty}\frac1t\int_0^t\boldsymbol u_sds \mbox{ is a Nash equilibrium. }
\end{equation*}
When the convexity property fails, the convergence statement in the previous result does not hold in general. It is well-known that in this case, the gradient descent algorithm may be stuck in local minima of the map $F_i$. 
In order to allow the gradient descent to escape from local minima traps, we follow the standard literature on numerical optimization by introducing a stochastic perturbation to the particle system \eqref{eq:sigma=0.equation}.
This leads naturally to the following Langevin-type stochastic particle system $\X^{\sigma,N}=(X^{i,N})_{1\le i\le N}$:
\begin{eqnarray}\label{eq:Part.Syst.Gen}
dX^{i,N}_t 
&=&  - \nabla_{\!x^i}F_i( \X^{\sigma,N}_t) dt 
                       +\sqrt{2\sigma}\; dW^i_t, \quad i = 1,\dots,N,
\end{eqnarray}
with parameter $\sigma\in(0,1]$.
Throughout this paper, we shall assume without further  mention that
\begin{equation}\label{X0L2}
\begin{array}{c}
\mbox{\it all processes defined through the particles system or the corresponding mean field limit}
\\
\mbox{\it are started from square integrable initial conditions.}
\end{array}
\end{equation} 
Under \ref{nablaFLip} and \ref{X0L2}, the SDE \eqref{eq:Part.Syst.Gen} has a unique square integrable strong solution $\X^{\sigma,N} = (X^{i,N})_{1\le i\le N}$ for all $\sigma>0$.

For later use, we start by showing that the solution of the interacting Langevin SDE \eqref{eq:Part.Syst.Gen} has a similar asymptotic behavior as \eqref{eq:sigma=0.equation}.

\begin{proposition} \label{Prop:Conv.N-Nash.to.avera}
  Assume that \eqref{nablaFLip} and \eqref{eq:convex.group} hold and that the family $\{F_i\}_{i=1,\dots,N}$ is $\ell_F$-displacement monotone.
  Let $\bx^*$ be a Nash equilibrium. 
  If $X^{i,N}_0 = u^{i,N}_0$, then putting
  $\overline{\X}^{\sigma,N}_t:= \frac1t\int_0^t\X^{\sigma,N}_sds$, we have:
  \begin{eqnarray*}
    \lim_{\sigma\to 0}\sup_{t\ge 0}\E\big[ \cW_2^2(\mu^N_{ \boldsymbol{u}_t},\mu^N_{{\X}^{\sigma,N}_t}) \big] = 0 \mbox{ and }
    \lim_{t\to\infty}\lim_{\sigma\to0}
    \E\big[\cW_2^2(\mu^N_{\bx^*}, \mu^N_{\overline{\X}_t^{\sigma,N}})] = 0,
  \end{eqnarray*} 
  where the two limits in the second claim commute.
\end{proposition}

\begin{proof}\
  By standard SDE estimates, it follows from the Lipschitz continuity $\nabla_{\!x_i}F_i$, $i=1,\ldots,N$ that $\E[|X^{i,N}_t|^2]\le \E[|X^{i,N}_0|^2] + CT(\sigma +1)$ for some $C>0$.

Denote $\Delta X^{i,N}:= X^{i,N} - u^{i,N}$, and apply It\^o's formula to $|\Delta X^{i,N}|^2$. Then it follows from the square integrability of the solution $\X^{\sigma,N}$ that the function $f(t):=\frac1N\sum_{i=1}^N\E[|\Delta X^{i,N}_t|^2]$ is differentiable with:
     \begin{align*}
       f'(t)
        &= 
        \sigma 
        -\frac1N\sum_{i=1}^N2\E\big[\Delta X^{i,N}_t
                                                       \!\cdot\!
                                                       \big(\nabla_{\!x_i}F_i(\X_t)
                                                               -\nabla_{\!x_i}F_i(\mathbf{u}_t)
                                                       \big)
                                                  \big]
       \\
       & \le \sigma 
               \!-\! \frac{{\color{violet} 2}\ell_F}{N}
                      \sum_{i=1}^N\!\E[|\Delta X^{i,N}_t|^2]
       = \sigma -{\color{violet} 2}\ell_F f(t),
 \end{align*}
by the $\ell_F$-displacement monotonicity of the family $\{F_i\}_{i\le N}$. 
  Multiplying both sides by $e^{2\ell_Ft}$, we have
$   \frac{d}{dt}\big(f(t)e^{2\ell_Ft} \big) \le \sigma e^{2\ell_Ft}.$
  Integrating on both sides from $0$ to $t$ and using $f(0) = 0$ yields
  \begin{equation}
  \label{eq:Deltaf}
    f(t) \le \frac{\sigma}{2\ell_F}\big( 1 - e^{-2\ell_Ft} \big)
  \end{equation}
and therefore
$$
\sup_{t\ge 0}\E\big[ \cW_2^2(\mu^N_{ \mathbf{u}_t},\mu^N_{\X^{\sigma,N}_t}) \big] \le \sup_{t\ge 0}f(t)=\frac{\sigma}{ 2\ell_F}\longrightarrow 0,
~~\mbox{as}~~\sigma\to 0.
$$
Moreover, using the triangle inequality we have that 
\begin{align*}
      \cW^2_2(\mu^N_{\bx^*}, \mu^N_{\overline \X^{\sigma,N}_t}) 
&      \le 
      2\cW_2^2(\mu^N_{\overline{\mathbf{u}}_t},\mu^N_{\overline{\X}^{\sigma,N}_t})
      +2\cW_2^2(\mu^N_{\bx^*}, \mu^N_{\overline{\mathbf{u}}_t})\\ 
      &\le  \frac2t\int_0^tf(s)ds 
              +\frac2N\sum_{i=1}^N\Big|x^{*,i} - \frac1t\int_0^tu^{i,N}_sds\Big|^2\\ 
      & \le \frac{\sigma}{\ell_F}\Big(1 + \frac{e^{-{\color{violet} 2}\ell_Ft} - 1}{{\color{violet} 2}t\ell_F}\Big)
+ 
      \frac2N\sum_{i=1}^N\Big|x^{*,i} - \frac1t\int_0^tu^{i,N}_sds\Big|^2
      .
\end{align*}
The first term on the right hand side converges to zero as $t\to\infty$ and $\sigma\to 0$, regardless of the order of the limits. In addition, as $t\to 0$, $\frac2N\sum_{i=1}^N\Big|x^{*,i} - \frac1t\int_0^tu^{i,N}_sds\Big|^2\longrightarrow 0$ as a Nash equilibrium exists due to the convergence of $\overline {\mathbf{u}}$ to a Nash equilibrium (see \cite[Proposition 3.1]{Awi-Hynd-Mawi23}), and the uniqueness of the Nash equilibrium implied by the displacement monotonicity of $\{F_i\}_{1\le i\le N}$ as we now show:

If $\bx,\by$ are two Nash equilibria, then by optimality of $x^i$ and $y^i$ respectively, we have $(z-x^i)\cdot \nabla_{x^i}F_i(\bx)\ge 0$ and $(z-y^i)\cdot \nabla_{x^i}F_i(\by)\ge0$ for all $z\in \R^d$.
This gives that $(y^i - x^i)\cdot \big[\nabla_{x^i}F_i(\bx) - \nabla_{x^i}F_i(\by)\big]\ge0$,  and therefore $0 \ge \sum_{i=1}^N(x^i \!-\! y^i)\!\cdot\!\big[ \nabla_{x^i}F_i(\bx) \!-\! \nabla_{x_i}F_i(\by)\big]$ implying that $\bx = \by$ by the strict displacement monotonicity of $\{F_i\}_{1\le i\le N}$.
\end{proof}

\begin{example}
  Consider the $2$-player game with $F_1(x,y) = x^2 + \sin(x)\cos(y)$ and $F_2(x,y) = y^2 + \sin(y)\cos(x)$.
  It is easily checked that $F_1$ and $F_2$ are convex in their first and second arguments, respectively.
  In addition, $\partial_xF_1$ and $\partial_y F_2$ are Lipschitz-continuous and $(F_1,F_2)$ is displacement monotone.
  To see the latter fact, let $x,x'$ and $y, y'$ be given.
  Put $g(x,y) = \cos(x)\sin(y)$.
  By the fundamental theorem of calculus and properties of cosine and sine functions, there are $A,B\in [-1,1]$ such that $g(x,y)-g(x',y') = A(x-x') + B(y-y')$.
  Then, it holds
  \begin{align*}
  &(x- x')(\partial_xF_1(x,y)-\partial_xF_2(x',y')) + (y-y')(\partial_yF_2(x,y) - \partial_yF_2(x',y'))\\ 
  &= 2|x-x'|^2+2|y-y'|^2+ (x - x')[ g(x,y) - g(x',y')] + (y - y') [ g(x,y) - g(x',y')]\\ 
  & = 2|x-x'|^2+2|y-y'|^2+A|x-x'|^2 + (A+B)(x-x')(y-y') + B|y-y'|^2
  \ge 0.
  \end{align*}
  In the present case the ODE system becomes
  \begin{equation*}
    \begin{cases}
      \dot u_{1,t} = -(2 u_{1,t} + \cos(u_{1,t}))\\
      \dot u_{2,t} = -(2 u_{2,t} + \cos(u_{1,t})).
    \end{cases}
  \end{equation*}
  Subtracting both sides, we see that invariant solutions of this system are vectors $(x,y)$ such that $x= y$ and $2x + \cos^2(x) =0$.
  This equation has a unique solution $x^*$.
  It belongs to $(-\frac12, 0)$.
  The point $(x^*,x^*)$ is also the unique Nash equilibrium.
\end{example}

\begin{remark}
  One can use the Lyapounov stability theory (see e.g. \cite{LaSalle60}) to provide conditions guaranteeing that the ODE system \eqref{eq:sigma=0.equation} admits invariant solutions.
  Such invariant solutions are Nash equilibria at least when $F_i$ is convex in $x^i$.
\end{remark}
Beside the convexity issue, the simulation of both $\boldsymbol{u}$ in \eqref{eq:sigma=0.equation} and $\X$ in \eqref{eq:Part.Syst.Gen} become prohibitively slow when the dimension $N$ becomes large.
In fact, we need to compute $N$ derivatives at each iteration of the algorithm.
The main objective of this paper is to introduce a particle-based method with convergence guaranties to approximate Nash equilibria in symmetric, large population games.

\section{Symmetric cost functions and mean field games}
\label{sec:ergodicity_of_the_langevin_system_and_nash_equilibria}

\subsection{Symmetric finite population game and mean field limit}

In this paper, we focus on the context of a symmetric game where the cost functions $F_i$ are defined through a single map
\begin{equation}\label{cost:sym}
F_i(\bx) := F\big(x_i, \mu^{N-1}_{\bx^{-i}}\big),
  ~~\mbox{for all}~~ 
  \bx=(x^1,\dots,x^N)\in(\R^d)^N,
  ~~\mbox{where}~~
F:\R^d\times \cP_2(\R^d)\longrightarrow \R.
\end{equation}
Then the particle system SDE \eqref{eq:Part.Syst.Gen} reduces in this case to:
\begin{eqnarray}
dX^{i,N}_t 
&=&  -\nabla_{\!x}F\big(X^{i,N}_t, \mu^{N-1}_{\X^{-i}_t}\big) dt 
                  + \sqrt{2\sigma}\; dW^i_t, \quad i = 1,\dots,N,
\label{eq:Part.Syst.Sym}
\end{eqnarray}
where we recall that $\X=\X^{\sigma,N}=\big(X^{1,N},\ldots,X^{N,N}\big)$.

\begin{remark}
  The reader may wonder whether we could consider the cost function $F_i(x^1,\dots, x^N) = F(x^i, \mu^{N}_{\bx})$ instead of the one we have elected to work with.
  Our choice of cost function will simplify the exposition without reducing the generality of the result.
  In fact, considering the cost function $F_i(x^1,\dots, x^N) = F(x^i, \mu^{N}_{\bx})$ would simply add the additional term $-\frac1ND_mF(X^{i,N}_t, \mu^N_{\X_t},X^{i,N}_t)$ to the drift of \eqref{eq:Part.Syst.Sym}.
  With this additional term, all the results of this work will remain true if $D_mF$ is a bounded and Lipschitz continuous function. 
\end{remark}

We next turn our attention to the description of the mean field game limit as in \citeauthor{cardaliaguet2010notes} \cite{cardaliaguet2010notes} and \citeauthor{Lacker-notes18}\cite{Lacker-notes18}.

\begin{definition}\label{def:MFG}
 A probability measure $m \in \cP_2(\R^d)$ is called a mean field equilibrium (MFE) of the mean field game defined by $F$ if 
  \begin{equation*}
    \text{supp}(m) = \argmin_{x\in \R^d}F(x,m). 
  \end{equation*}
 \end{definition}

When the cost function $F$ is a lower semicontinuous map on $A\times \cP(A)$ for some convex and compact set $A$, \cite{cardaliaguet2010notes} and \cite{Lacker-notes18} show that there exists an MFE and that uniqueness holds when $F$ satisfies the Lasry-Lions monotonicity condition, see \eqref{eq:LL} below.
In addition, they showed that the sequence of empirical measures formed by Nash equilibria subsequentially converges to an MFE. This justifies the relevance of mean field games as an approximation tool for (symmetric) large population games.

Motivated by this, we consider in this paper the natural mean field limit of the particle system \eqref{eq:Part.Syst.Sym} which reduces to the McKean-Vlasov Langevin equation
\begin{equation}
\label{eq:McVLangevin2}
  dX_t = -  \nabla_{\!x}F(X_t,m_{X_t}) dt + \sqrt{2\sigma}\; dW_t \quad\text{with } m_{X_t} := \P\circ X^{-1}_t.
\end{equation} 
The starting point of our paper is to establish conditions which guarantee that this McKean-Vlasov Langevin SDE has an invariant measure $m^\sigma$ for all $\sigma>0$. The family $(m^\sigma)_{\sigma>0}$ will be shown to be tight with accumulation points coinciding with the MFE.
Under either notion of monotonicity in Definition \ref{ass.convex} below, we shall establish uniqueness of MFE, and we shall further justify the approximation of the MFE of the particle system \eqref{eq:Part.Syst.Sym} with concrete error bounds.
 
We note that our existence results are akin to the Laplace principle used in optimization theory, see e.g. \citeauthor{Hwang80} \cite{Hwang80}.

\subsection{Assumptions and examples}

We recall that Condition \ref{X0L2} is always in force, and we will also assume throughout that $\nabla_{\!x}F$ is Lipschitz as the analogue of \ref{nablaFLip} in the current symmetric setting. For completeness, these conditions are included in the following assumptions.
\begin{assumption}
\label{Ass.MFG}
The map $F:\R^d\times\cP_2(\R^d)\longrightarrow\R$ is continuously differentiable with bounded Lions' derivative $D_mF$, and:
\begin{itemize}
\item[(i)] $F$ is $\cW_p-$Lipschitz in $m$, uniformly in $x$, for some $p\in [1,2)$ i.e. $\sup_{x\in \R^d}\big|F(x,m) - F(x, m')\big| \le C_F\cW_p(m,m')$;
  \item[(ii)] $\nabla_{\!x}F \in C^1(\R^d\times\cP_2(\R^d))$ is Lipschitz and satisfies for some constants $\alpha>0$ and $C_1, C_2\ge0$:
  \begin{equation*}
    2x\cdot \nabla_{\!x}F(x,m) \ge \alpha|x|^2 - C_1 + C_2(|x|^2-|m|_2^{2}),
    ~\mbox{for all}~
    (x, m )\in \R^d\times \cP(\R^d).
  \end{equation*}
\end{itemize}
\end{assumption}
Assumption \ref{Ass.MFG} (i) 
 guarantees existence of unique strong solutions of \eqref{eq:Part.Syst.Sym} and \eqref{eq:McVLangevin2} for any square-integrable initial conditions, see again our standing condition \eqref{X0L2}. 

Assumption \ref{Ass.MFG} (ii) is a standard dissipativity condition ensuring the existence of an invariant measure for our McKean-Vlasov SDE. 

\medskip
The following notions will be often used.

\begin{definition}
\label{ass.convex}
A function $V:\R^d\times \cP(\R^d)\to \R$ is called 
\begin{itemize}
\item Lasry-Lions monotone if for all $m,m'\in \cP(\R^d)$:
  \begin{equation}
  \label{eq:LL}
  \Gamma_{_{\!\rm LL}}(m,m')
  :=
  \int_{\R^d}[V(x,m) - V(x,m')](m-m')(dx)\ge 0,
  \end{equation}
and strictly Lasry-Lions monotone if in addition $\Gamma_{_{\!\rm LL}}(m,m')>0$ for all $m\neq m'$;
\item Displacement monotone  if for all $m,m'\in \cP(\R^d)$ and all $\pi\!\in\!\Pi(m,m')$:
\begin{equation}
\label{eq:DM}
\hspace{5mm}
  \Gamma_{_{\!\rm DM}}(m,m'\!,\pi)
  \!:=\!\!\int_{\R^d}[\nabla_{\!x}V(x,m)\!-\!\nabla_{\!x}V(x', m')]
                           \!\cdot\! (x \!-\! x')
                           \pi(dx,dx') \ge 0,
\end{equation}
and strictly displacement monotone if in addition $\Gamma_{_{\!\rm DM}}(m,m',\pi)>0$ for all $m\neq m'$ and $\pi\in\Pi(m,m')$;
\item $L-$Displacement monotone, for some $L\in\R$, if the map $(x,m)\longmapsto V(x,m)-\frac12 L|x|^2$ is Displacement monotone.
\end{itemize}
\end{definition}

In general, neither one of Lasry-Lions monotonicity and displacement monotonicity implies the other.
\begin{example} We provide here some examples of maps $F$ satisfying our conditions.

\medskip
\noindent {\rm (i)} Let $F(x,m) = \int_{\R^d}\varphi(x-y)m(dy) + g(x)$ for some bounded odd function $\varphi:\R^d\longrightarrow\R$ with bounded second derivative and a sufficiently convex function $g$.
It can be checked, see e.g. \citeauthor{carmona2018probabilisticI} \cite[Section 3.4.2]{carmona2018probabilisticI} that this function is Lasry-Lions monotone.
Observe that if $\varphi(\cdot)+g(\cdot) - \frac{\ell}{2}|\cdot|^2$ is convex, then $F$ is $\ell$-convex.

\medskip
\noindent {\rm (ii)} Let $F(x,m) = \phi(x)\int_{\R^d}\phi(y)m(dy) + \frac12|x|^2$ for some smooth bounded function $\phi:\R^d \longrightarrow \R$ with bounded first and second derivatives.
Then, the function $F$ is Lasry-Lions monotone. If $\phi$ is such that $x\cdot \nabla \phi(x)\int_{\R^d}\phi(y)m(dy) \ge a(1 + |x|^2 - |m|^2)$ for some $a\ge 0$, then $F$ satisfies Assumption \ref{Ass.MFG}.
Moreover, if there is $\ell \in \R$ such that $\nabla^2 \phi\int \phi(y) m(dy)+ {\color{violet}I} \ge \ell {\color{violet}I}$, then $F$ is $\ell$-convex.

\medskip
\noindent {\rm (iii)} Given a bounded function $\phi \in C^2(\R^d)$ with $-CI_{\mathrm{d}}\le \nabla^2\phi<0$, consider the function $F(x,m) = C|x|^2 - (\phi\ast m)(x)$ where $\phi\ast m$ is the convolution of $\phi$ and $m$.
This function satisfies Assumption \ref{ass.convex}. Moreover,  \citet[Remark 2.1]{Meszaros-Mou24} shows that this function is not Lasry-Lions monotone but is Displacement monotone.
\end{example}

Notice that the Lipschitz constant of $\nabla_{\!x} F$ in $m$ in both examples (i) and (ii) can be arbitrarily large in contrast to most of the literature on contractivity and uniform in time propagation of chaos, see e.g. \citet{Hu-Ren-Sis-Sz21}, \citet{Eberle-Guillin20,Eberle-Guillin-Zimmer19}.

\subsection{Main results}
Our first result provides a construction of mean field equilibria using the McKean-Vlasov Langevin equation and derives uniqueness as a consequence of either notion of strict monotonicity in Definition \ref{ass.convex}

\begin{theorem}
\label{thm:Laplace.principle} 
 Under Assumption \ref{Ass.MFG}, we have the following statements.

\medskip
\noindent {\rm (i)} For every $\sigma>0$, the McKean-Vlasov SDE \eqref{eq:McVLangevin2} has an invariant measure $m^\sigma$; the family $(m^\sigma)_{\sigma>0}$ is tight and any of its accumulation points is an MFE.

\medskip
\noindent {\rm (ii)} Assume in addition that either one of the following conditions holds:
    \begin{itemize}
    \item[(a)] $F$ is strictly Lasry-Lions monotone,
    \item[(b)] or $F$ is strictly displacement monotone.
    \end{itemize}
Then, \eqref{eq:McVLangevin2} has a unique invariant measure $m^\sigma$ converging to the unique MFE $m^0$. 

\medskip
\noindent {\rm (iii)} Let $m^0$ be a MFE and consider a family of independent r.v. $\{X^i,i\ge1\}$ with law $m^0$. Let $\X^N:=(X^i)_{i\le N}$ and $A:=\{\lim_{N\to\infty}\mu^N_{\X^N}= m^0,~\mbox{weakly}\}$. Then there is a sequence $\varepsilon^N\to0$ such that, for all $\omega\in A$, the strategy $(X^1,\ldots,X^N)(\omega)$ is an $\varepsilon^N$-Nash equilibrium of the $N$-player game.
\end{theorem}

Our next objective is to establish the uniform in time propagation of chaos of our particle system. In contrast to the existing literature, our approach leverages monotonicity conditions rather than relying on the smallness of interactions. 

\begin{theorem}
\label{thm:contracti.Unin.prop.chaos}
  Let Assumption \ref{Ass.MFG} hold and further assume that 
  $F$ is $\ell_F$-Displacement monotone for some $\ell_F>0$.
  Then: 
  \begin{itemize}
    \item[(i)] The McKean-Vlasov SDE \eqref{eq:McVLangevin2} has a unique invariant measure $m^\sigma$ whose deviation from the marginal law $m_{X_t}$ is estimated by: 
\begin{equation*}
\sup_{t>0}
e^{2 \ell_F t}\,
\cW_2^2(m_{X_t}, m^\sigma)<\infty.
\end{equation*} 
\item[(ii)] Assume further that $\nabla_{\!x}F$ is twice differentiable in $m$ with 
\begin{equation}\label{eq:smoothness.F}
\mbox{maps}~\delta_m\nabla_{\!x} F, \delta^2_{mm}\nabla_{\!x} F
~\mbox{bounded and Lipschitz in}~x,
\end{equation}
and let $\sup_{N\ge 1}\sup_{i\le N}N\E[|X^{i,N}_0 - X_0|^2]<\infty$. Then
\begin{equation}
\label{eq:PoC.LL.thm}
\sup_{N\ge 1}\sup_{t\ge 0} N\E[|X^{1,N}_t - X_t|^2]<\infty. 
\end{equation}
\end{itemize}
\end{theorem}

\begin{remark}
The $\ell_F$-displacement monotonicity condition $\ell_F>0$ can be weakened by allowing $\ell_F\in \R$ and adding $\sigma\nabla U$ to the drift of the McKean-Vlasov SDE \eqref{eq:McVLangevin2}, for some $\ell_U$-(semi)convex confinement potential $U:\R^d\longrightarrow\R$, with $\ell_U\in \R$ and appropriate dissipative condition. 
In this setting, the statement of Theorem \ref{thm:contracti.Unin.prop.chaos} holds true with exponential speed at the rate $\ell_F+\sigma\ell_U>0$. 
This condition exhibits the tradeoff needed between the monotonicity of $F$ and the convexity of $U$ to guarantee contractivity and uniform in time propagation of chaos under the displacement (semi)monotonicity condition. 
\end{remark}

We now state approximation results of the MFE by our particles systems. 

\begin{theorem} \label{thm:contracti.Unin.prop.chaos-iv}
Let Assumption \ref{Ass.MFG} hold. Assume further that $F$ is 
$\ell_F$-displacement monotone for some $\ell_F>0$.
Let $m^0$ be the mean field equilibrium.
Then: 
\begin{itemize}
      \item[(i)] there is some constant $C>0$ such that 
    \begin{equation*}
        \cW_2(m_{X_t}, m^0) \le Ce^{-\ell_Ft} + C\frac{\sigma}{\ell_F}
        ~~\mbox{for all}~~
        t,\sigma>0;
    \end{equation*}
      \item[(ii)]  if $\nabla_xF$ is twice differentiable in $m$ and satisfies \eqref{eq:smoothness.F}, then
    \begin{equation*}
          \cW_2(m_{\X^{\sigma,N}_t}, m^0) \le \frac{C}{N} 
    + Ce^{-\ell_Ft} + C\frac{\sigma}{\ell_F}
           ~~\mbox{for all}~~
           t,\sigma>0,N\ge1.
      \end{equation*}
\end{itemize}
\end{theorem}
Our final result provides the link between the finite population and the mean field game.
\begin{theorem} \label{thm:contracti.Unin.prop.chaos-iii}
Let the conditions of Theorem \ref{thm:contracti.Unin.prop.chaos-iv}
hold, and let $\bx$ be a Nash equilibrium of the $N$-player game. Then:
  	\begin{itemize}
  	\item[(i)] with $\delta_{N,p}$ being the Fournier-Guillin rate \eqref{deltaN} and $p<2$, the constant in Assumption \ref{Ass.MFG},
	$$
	\sup_{N\ge 1}\frac1{\delta_{N,p}}\cW_p(\mu^N_\bx,m^0)<\infty;
	$$
  	\item[(ii)] there are constants $t_0\ge0$ and $C>0$ such that for sufficiently large $N$ and small $\sigma$:
    	\begin{equation*}
      \sup_{t\ge t_0}\P\big[\cW_1(\mu^N_{\X_t}, \mu_{\bx}^N)>r\big]   \le 
      2\e^{-\frac{\ell_F}{10N\sigma}(\frac{r^2}{5}-CNr\sqrt{\frac{\sigma}{\ell_F}})}  
      ~\mbox{for all}~r> 0.
	\end{equation*}
	\end{itemize}
\end{theorem}
In the last statement, we used the \citeauthor{fournier2015rate} rate \cite[Theorem 1]{fournier2015rate}:
    \begin{eqnarray}\label{deltaN}
      \delta_{N,p} 
      &:=&
      \begin{cases}
      N^{-\frac12} + N^{-\frac{2-p}{2}},\quad \text{if $p>\frac{d}{2}$ and } p\neq 1 \\
       N^{-\frac12}\log(1 +N) + N^{-\frac{2-p}{2}}\quad \text{if } d=2p\text{ and } p\neq 1\\
          N^{-\frac{p}{d}} + N^{-\frac{2-p}{2}}\text{ if } p\in(0,\frac{d}{2})\text{ and } \frac{d}{d-p}\neq 2.
      \end{cases}
   \end{eqnarray}
   Observe that under the conditions on 
Theorem \ref{thm:contracti.Unin.prop.chaos-iii} the function $x\mapsto F(x,\cdot)$ is convex.
This limitation is not too surprising because we are seeking strict Nash equilibria.
In fact, as is well-known, most non-convex/quasiconvex games will not admit Nash equilibria, see e.g. \citet{Dask26}.
An interesting subject of future research is to understand the extend to which the method developed in this work applies to the approximation of randomized Nash equilibria or to other (weaker) notions of equilibria.


\begin{remark}\label{rem:Dis-LL}
If the cost function $F$ is rather assumed to be Lasry-Lions monotone, we can recover the quantitative results above.
In fact, it is proved in \citet[Subsection 2.3]{Gangbo-Mou-Zhan-Me22} that an $\ell-$convex Lasry-Lions monotone map is $2\ell-$Displacement monotone, see also \citet{Gang-Mes22} for more details and examples.
\end{remark}

\begin{remark}
  By following the same argument as in the proof of Theorem \ref{thm:contracti.Unin.prop.chaos}, and exploiting the Lipschitz property of $\nabla_{\!x}F$ in Assumption \ref{Ass.MFG} (ii), the additional regularity conditions on $\nabla_{\!x}F$ in Theorem \ref{thm:contracti.Unin.prop.chaos} (ii) can be by-passed when $(X^{i,N}_0)_{i=1,\dots,N}$ are independent copies of $X_0$ and $\sup_t\E[|X_t|^q]<\infty$ for some $q>2$. This would leads to the (slower) convergence rate in Wasserstein distance:
  \begin{equation}
    \label{eq:UinTPoC1}
      \E\big[\cW_2\big(\mu^N_{\X_t}, m_{X_t}\big)\big] 
      \le 
      \frac{C}{\sqrt{N}}
       +\frac{C}{\ell_F}\Big(1 - e^{-\ell_F t}\Big) \delta_{N,q}^2 , 
  \end{equation}
for some constant $C>0$ independent of $N,t$.
\end{remark}
 
Our last result 
provides an alternative contractivity property that is weaker than that of Theorem \ref{thm:contracti.Unin.prop.chaos} (ii), at the expense of a weaker monotonicity condition on $F$.

\begin{proposition}
\label{thm:contracti.DMweak}
  Let Assumption \ref{Ass.MFG} {\rm (ii)} be satisfied.
 Assume that, for some $c_F>0$,  the function $F$ satisfies for all $m,m'\in \cP_2(\R^d)$ and $\pi\!\in\!\Pi(m,m')$:
  \begin{align}
  \label{eq:weakDM}
  \hspace{-3mm}
    \int\!(x\!-\!x')\!\cdot\!\big[\nabla_{\!x}F(x,m) \!-\! \nabla_{\!x}F(x'\!,m')
                                   \big]\pi(dx,\!dx') 
    \ge 
    c_F\;
    \frac{\int \!|x\!-\!x'|^2\pi(dx,\!dx')}
           {1+|m|_1 + |m'|_1}.
  \end{align}
Then $\sup_{t\ge 0} \sqrt{t}\,\cW_2(m_{X_t},m^\sigma)<\infty$.
\end{proposition}
\begin{remark}
\label{rem:potential.games}
  A particularly interesting situation arises when the function $F$ is the linear derivative of a function $\cF:\cP(\R^d)\to \R$.
That is, when it holds
\begin{equation*}
  F(x,m) = \delta_m \cF(m,x).
\end{equation*}
In this case, one can observe that any minimizer $m$ of the potential $\cF$ yields an MFE of the game with cost $F$, and the equation \eqref{eq:McVLangevin} becomes the McKean-Vlasov Langevin dynamics
\begin{equation}
\label{eq:MFLangevin}
  dX_t = -  D_m\cF(m_{X_t}, X_t) dt + \sqrt{2\sigma} dW_t.
\end{equation}
Contractivity as well as uniform in time propagation of chaos for this particle system have been extensively analyzed notably motivated by deep leaning applications.
Some recent works include \citet{Hu-Ren-Sis-Sz21} \cite{Kaz-Ren-Tan-Yan24} and  \citet{Eberle-Guillin20,Eberle-Guillin-Zimmer19}.
\citet[Theorem 2.1]{Hu-Ren-Sis-Sz21} assumes that $\cF$ is convex and that $D_m\cF$ is sufficiently regular to obtain a \emph{qualitative} convergence result  of $m_{X_t}$ to the invariant measure. With extra conditions, notably when the Lipschitz constant of $D_m\cF$ in the measure argument is small enough, \citet{Eberle-Guillin20,Eberle-Guillin-Zimmer19} derive exponential contractivity and uniform in time propagation of chaos.
It is interesting to observe that when the function $\cF$ is convex, its linear derivative $F$ is Lasry-Lions monotone; and when $\cF$ is displacement convex, its linear derivative $F$ is displacement monotone.
In general, as stated in the introduction, smallness conditions on the Lipschitz constant in the measure argument are omnipresent in the literature on contractivity and uniform in time propagation of chaos, except for a few recent results focussing on \eqref{eq:MFLangevin} with convex $\cF$; see \cite{SongboWang24,Kook.etal24,Monmarche24}.
Thus, our method and results give an alternative to existing works on the topic allowing to avoid small interaction thanks to monotonicity. 
\end{remark}

\section{Mean field games and Ergodicity of Langevin dynamics}
\label{sec:mean_field_games_and_ergodicity_of_langevin_dynamics}
This section is dedicated to the proof of Theorem \ref{thm:Laplace.principle}.
We begin by an integrability property.
\begin{lemma}
\label{lem:integrability}
Under Assumption \ref{Ass.MFG}, the solution $X$ of the mean field SDE \eqref{eq:McVLangevin2} satisfies
  \begin{equation*}
    \sup_{0<\sigma\le 1}\sup_{t\ge0}\E[|X_t|^2]<\infty.
  \end{equation*}
\end{lemma}
\begin{proof}
   Applying It\^o's formula and Assumption \ref{Ass.MFG}.$(ii)$, we have 
   \begin{align*}
     \frac{d}{dt}\E[|X_t|^2] &= -2\E\big[X_t\cdot \nabla_xF(X_t,m_t) \big] + \sigma
                    \le - \alpha\E[|X_t|^2] +C_1 + \sigma.
   \end{align*}
   We apply Grownall's inequality as in the proof of Proposition \ref{Prop:Conv.N-Nash.to.avera} to get
   \begin{align*}
     \E[|X_t|^2] \le \E[|X_0|^2]e^{-\alpha t} + \frac{C_1 +\sigma}{\alpha}\Big( 1 - e^{- \alpha t} \Big)
   \le\E[|X_0|^2] + \frac{C_1 +1}{\alpha},
   ~t\ge 0,~\sigma\in(0,1].
   \end{align*}
\end{proof}
\begin{proof}[Proof of Theorem \ref{thm:Laplace.principle} {\rm (i)}]
    Since $\nabla_{\!x}F$ is Lipschitz continuous, it follows for instance from \citeauthor{carmona2018probabilisticI} \cite[Theorem 4.21]{carmona2018probabilisticI} that for any $X_0\in \L^2(\Omega)$ the SDE \eqref{eq:McVLangevin2} admits a unique strong solution $X$ satisfying
  \begin{equation*}
    \E\big[\sup_{0\le t\le T}|X_t|^2\big] <\infty\quad \text{for all } T>0.  
  \end{equation*}     
By \citet[Theorem 2.2]{Zhang_Shao-Qin23}, Assumption \ref{Ass.MFG} (ii) also guarantees that for every $\sigma>0$, $X$ admits an invariant measure $m^\sigma$.
  Observe that any stationary measure $m^\sigma$ of \eqref{eq:McVLangevin2} is absolutely continuous with respect to Lebesgue measure and is a solution of the stationary nonlinear Fokker-Planck equation $\nabla \cdot \big( m(x)\nabla_{\!x} F(x, m) \big) + \sigma\Delta m(x) = 0$, see e.g. \citet{Barbu-Roeckner20}.
  Integrating this equation on both sides gives $m^\sigma(x)\cdot\nabla_xF(x,m^\sigma)/\sigma = -\nabla m(x) +c $ for some constant $c\in \R$.
  This ODE is uniquelly solved by
  \begin{equation}
  \label{eq:station.sigma}
    m^\sigma(dx) = \frac{1}{Z^\sigma}e^{-\frac{F(x,m^\sigma)}{\sigma}}dx\quad \text{with}\quad Z^\sigma := \int_{\R^d} e^{-\frac{F(x,m^\sigma)}{\sigma}} dx.
  \end{equation}

\medskip

\noindent {\bf 1.} 
By the ergodic theorem, see \cite[Corollary 2.5.9]{Kallenberg21}, the measure $m^\sigma$ is the weak limit of the Ces\'aro average $\eta_t:=\frac{1}{t}\int_0^tm_{X_s}ds$. Then, 
$$
\int_{\R^d}\!\!|x|^2m^\sigma(dx)
=\lim_{n\to\infty}\!\!\!\uparrow\!\!\int_{\R^d}\!\!(n\wedge|x|^2)m^\sigma(dx)
=\lim_{n\to\infty}\!\!\!\uparrow\lim_{t\to\infty} \frac1t\E\!\!\int_0^t(n\wedge|X_t|^2) dt
\le
\sup_{t\ge 0}\sup_{\sigma\in(0,1]}\E[|X_t|^2]
=:C,
$$
with $C<\infty$ by Lemma \ref{lem:integrability}.
As $m^\sigma(\{|x|>k\}) \le \frac{1}{k}\int_{\R^d}|x|^2m^\sigma(dx)$ for all $k>0$, this shows that the family $(m^\sigma)_{\sigma>0}$ is tight.

\noindent {\bf 2.} Let $m^0$ be the limit of an arbitrary converging subsequence, still denoted $(m^\sigma)_{\sigma>0}$, and let us show that $m^0$ is a MFE. We first  construct $a_n\in \R$, a sequence such that
  \begin{equation}\label{an}
  a_n\searrow \inf_{x\in \R^d}F(x,m^0)
  ~~\mbox{and}~~
  m^0\{x:F(x,m^0) = a_n\} = 0,
  ~\mbox{for all}~n\ge 1.
  \end{equation}
To this end, let $a\in \R$, $k\in \N$ and put $E_a:=\{x\in \R^d: F(x,m^0) = a\}$ and $A_k:=\{a: m^0(E_a)\ge1/k\}$.
  The sets $E_a$ are pairwise disjoint and for each $k$, the set $A_k$ is finite.
  In fact, if any $A_k$ were infinite, then we would have
  \begin{equation*}
    m^0(\R^d) \ge m^0(\cup_{a\in A_k}E_a) = \sum_{a \in A_k}m^0(E_a) \ge \sum_{a\in A_k}\frac1k = \infty,
  \end{equation*}
  which is a contradiction.
  Thus, the set $A:=\{ a:m^0(E_a)>0 \} = \cup_{k=1}^\infty A_k$ is countable.
  By definition of infimum, there is $b_n\downarrow \inf_{x\in \R^d}F(x,m^0)$ and since $A$ is countable, we can slightly modify $b_n$ to $a_n$ (if needed) such that $a_n \notin A$.
  Thus $a_n$ satisfies the desired properties, i.e. $m^0(E_{a_n}) = 0$ and $a_n\downarrow \inf_{x\in \R^d}F(x,m^0)$.

By continuity of $F(\cdot,m^0)$, we have that $\partial\{x: F(x, m^0) \!>\! a_n\} \subset \{x: F(x, m^0) = a_n\}$, which implies that $m^0\big\{ \partial\{x: F(x, m^0) \!>\! a_n\} \big\} = 0$.
  Thus, by Portmanteau theorem, see e.g. \citeauthor{billingsley1999convergence} \cite[Theorem 2.3]{billingsley1999convergence},
  \begin{equation}\label{msigmam0}
    m^\sigma\big\{x: F(x, m^0) \!>\! a_n\big\} \xrightarrow [\sigma\to 0]{} m^0\big\{x: F(x,m^0)\!>\! a_n\big\}.
   \end{equation} 
   Moreover, we can bound the left hand side above as
   \begin{align*}
     \int_{\R^d}e^{-\frac{F(x,m^\sigma) }{\sigma}}
     \mathbf{1}_{\{F(x,m^0)>a_n\}}dx
     &= e^{-\frac{a_n}{\sigma}}
            \int_{\R^d}e^{-\frac{F(x,m^\sigma) - F(x,m^0)}{\sigma}}
                            e^{-\frac{F(x,m^0) - a_n}{\sigma}}
                            \mathbf{1}_{\{F(x,m^0)>a_n\}} dx
     \\ 
     &
     \le e^{-\frac{a_n}{\sigma}+\frac1\sigma\Vert F(\cdot,m^\sigma) - F(\cdot,m^0)\Vert_\infty}
     \le
     e^{-\frac{a_n}{\sigma}+\frac{C}\sigma\cW_p(m^\sigma,m^0)},
   \end{align*}
by the Lipschitz continuity of $F$ in Assumption \ref{Ass.MFG} (ii), and similarly:
   \begin{align*}
     Z^\sigma &\ge e^{-\frac{a_n}{\sigma}
                                  -\frac1\sigma\Vert F(\cdot,m^\sigma) - F(\cdot,m^0)\Vert_\infty} 
     \!\!\!\int_{\R^d}\!\! e^{\frac{a_n-F(y,m^0)}{\sigma}}dy
                     \ge e^{-\frac{a_n}{\sigma}-\frac{C}\sigma\cW_p(m^\sigma,m^0)}
                             \!\!\!\int_{\R^d}\!\!e^{\frac{a_n-F(y,m^0)}{\sigma}}
                                              \mathbf{1}_{\{a_n>F(y,m^0)\}}dy.
   \end{align*}
We then deduce that:
   \begin{align}
   \label{eq:m.sigma.conv.proof}
      m^\sigma\big\{x: F(x, m^0) \ge a_n\big\} 
      & \le \frac{e^{\frac{2C}{\sigma}\cW_p(m^\sigma,m^0)}}
             {\int_{\R^d}\mathbf{1}_{\{a_n>F(y,m^0)\}}e^{-\frac{F(y,m^0) - a_n}{\sigma}} dy}.
   \end{align}
Recall that $a_n>\inf_{x\in \R^d}F(x,m^0)=:\kappa$, and set $4\delta_n = a_n - \kappa$. By definition of infimum, there is $x_0$ such that $F(x_0,m^0) < \kappa + 2\delta_n = a_n - 2\delta_n$, and by continuity of $F(\cdot,m^0)$, there is a neighborhood $O$ of $x_0$ such that $F(x,m^0) < a_n - \delta_n$ for all $x \in O$, i.e. $O\subset\{F(\cdot,m^0) < a_n - \delta_n\}$.
Thus, denoting by $\nu$ the Lebesgue measure, $\nu[F(\cdot,m^0)<a_n-\delta_n] \ge \nu[O]>0$.
As $(m^\sigma)_{\sigma}$ is bounded in second moment and converges weakly to $m^0$, it follows that the convergence holds in  $p$-Wasserstein metric with $p<2$, see e.g. \citet[Theorem 6.9]{villani2009optimal}. 
  Thus, for $\delta_n>0$ as above, there is $\sigma_0$ such that for $\sigma<\sigma_0$ we have $\cW_p(m^\sigma,m^0)\le \delta_n/2C$.
  This shows by \eqref{eq:m.sigma.conv.proof} that for all $\sigma<\sigma^0$ we have
  \begin{equation*}
    m^\sigma\{ x: F(x, m^0)\ge a_n \} \le \frac{1}{\int_{\R^d}\mathbf{1}_{\{a_n>F(y,m^0)+\delta_n\}}e^{\frac{a_n - F(y,m^0) - \delta_n}{\sigma}} dy}  \xrightarrow [\sigma\to 0]{}0.
  \end{equation*}

In view of \eqref{msigmam0}, this shows that $m^0\big\{x: F(x,m^0)\!>\! a_n\} = 0$ and, recalling the definition of $(a_n)_n$ in \eqref{an}, it follows from the monotone convergence theorem that:
   \begin{align*}
    0 & = \lim_{n\to\infty}m^0\big\{x: F(x,m^0)> a_n\big\}
          = m^0\Big\{x: F(x,m^0)> \inf_{y\in \R^d}F(y, m^0)\Big\}.
   \end{align*}
i.e. $m^0\big[\argmin F(\cdot,m^0)\big] = 1$ meaning that $m^0$ is an MFE in the sense of Definition \ref{def:MFG}.  
\end{proof}

\begin{proof}[Proof of Theorem \ref{thm:Laplace.principle} (ii)]
\emph{Case 1: $F$ is strictly Lasry-Lions monotone.} In this case, uniqueness of the MFE is the same as in \citeauthor[Theorem 4.6]{Lacker-notes18} \cite[Theorem 4.6]{Lacker-notes18}.

Let us now prove uniqueness of the invariant measure for each $\sigma>0$.
    Any invariant measure $m^\sigma$ satisfies \eqref{eq:station.sigma}.
    By \citeauthor*[Proposition 2.5]{Hu-Ren-Sis-Sz21} \cite[Proposition 2.5]{Hu-Ren-Sis-Sz21}, $m^\sigma$ is the minimizer of the map $\mu\longmapsto\Phi(\mu,m^\sigma)$, where (recall that we denote by $\nu$ the Lebesgue measure)
  \begin{equation*}
    \Phi(\mu,m) :=
    \int_{\R^d} F(x,m^\sigma) \mu(dx) + \sigma \mathrm{H}(\mu|\nu),
    ~~\mu\in\cP_2(\R^d),
    \end{equation*}
and $\mathrm{H}(\mu|\nu) 
    := \mathbf{1}_{\{\mu\ll \nu\}} \int\log(\frac{d\mu}{d\nu})\mu(dx)
        +\infty \mathbf{1}_{\{\mu\not\ll \nu\}}$ is the Kullback-Leibler divergence. Suppose to the contrary that there are two probability measures $m,m'$ which satisfy the last minimality property. As the minimality of $m$ and $m'$ imply respectively that $\Phi(m,m)-\Phi(m',m)\le 0$ and $\Phi(m',m')-\Phi(m,m')\le 0$, we see that
\begin{eqnarray*}
0 &\ge& [\Phi(m,m)-\Phi(m',m)]
              +[\Phi(m',m')-\Phi(m,m')]
   \\
   &=& \int_{\R^d} F(x,m) m(dx) + \sigma \mathrm{H}(m|\nu)
           -\int_{\R^d} F(x,m) m'(dx) - \sigma \mathrm{H}(m'|\nu)
   \\
   && +\int_{\R^d} F(x,m') m'(dx) + \sigma \mathrm{H}(m'|\nu)
         -\int_{\R^d} F(x,m') m(dx) - \sigma \mathrm{H}(m|\nu)
  \\
   &=& \int_{\R^d} \big(F(x,m)-F(x,m')\big)(m-m')(dx) ,
\end{eqnarray*}
which in view of the strict Lasry-Lions monotonicity implies that $m = m'$. 
By Theorem \ref{thm:Laplace.principle} {\rm (i)}, this shows that the whole sequence $(m^\sigma)_{\sigma>0}$ converges to the unique MFE $m^0$.

\medskip

\noindent \emph{Case 2: $F$ is strictly displacement monotone.} Recall from \citeauthor[Lemma 2.6]{Gangbo22} \cite[Lemma 2.6]{Gangbo22} that this condition implies that $F(\cdot,m)$ is convex. Then, for any two equilibria $m$ and $m'$, we have $\text{supp}(m) \subseteq \{x: \nabla_{\!x}F(x,m) = 0\}$ and $\text{supp}(m') \subseteq \{x: \nabla_{\!x}F(x,m') = 0\}$.
    This shows that 
    \begin{equation*}
      \int_{\R^d}\int_{\R^d}(x - y)\cdot (\nabla_{\!x}F(x,m) - \nabla_{\!x}F(y,m'))m(dx) m'(dy) = 0,
    \end{equation*}
which implies that $m = m'$, by strict displacement monotonicity. 

Let us now prove uniqueness of the invariant measure for each $\sigma>0$. Let $m, m' \in \cP_2(\R^d)$ be two invariant measures, and let $X$ and $X'$ be two solutions of \eqref{eq:McVLangevin2} with respective laws $m$ and $m'$.
Notice that the sequence of probability measures $\eta_t(\cdot):=\frac1t\int_0^t\cL(X_s,X_s')(\cdot)ds\in \cP_2(\R^d\times \R^d)$ is tight as:
  \begin{align*}
     \int_{\R^d\times \R^d}(|x|^2+|x'|^2)\eta_t(dx,dx') &= \frac1t\E\Big[\int_0^t|X_s|^2 + |X'_s|^2ds\Big]
     = \E\big[|X_0|^2 + |X_0'|^2\big]<\infty.
   \end{align*} 
   Then, it follows from Prokhorov's theorem that, up to a subsequence, $(\eta_t)_{t\ge0}$ converges to a probability measure $\eta\in \cP_2(\R^d\times \R^d)$.
   By construction, $\eta$ is an invariant measure of $(X_t, X'_t)_{t\ge0}$.
   Therefore, using It\^o's formula and the displacement monotonicity condition, we see that
   $$
     \E[|X_0 \!-\! X'_0|^2] 
     = \E[|X_t \!-\!X_t'|^2]
     = \E[|X_0 \!-\! X'_0|^2] \!-\!\!\int_0^t\!\!\E\Big[(X_s \!-\! X'_s)
                                                                              \!\cdot\!
                                                                              \big(\nabla_{\!x}F(X_s, m) \!-\! \nabla_{\!x}F(X'_s, m')\big)
                                                                      \Big]ds,
   $$
which implies that $m=m'$ by the strict displacement monotonicity condition.
\end{proof}

\begin{proof}[Proof of Theorem \ref{thm:Laplace.principle}$(iii)$]
    As the law of $X^i$ is $m^0$, we have $X^i(\omega) \in \text{supp}(m^0)$ $\P$-a.s. for every $i=1,\dots,N$. Then, since $m^0$ is a MFE, we have for all $y\in \R^d$ that
    \begin{align*}
      F(X^i, \mu^{N-1}_{\X^{-i}}) - F(y, \mu^{N-1}_{\X^{-i}})
      & \le F(X^i, \mu^{N-1}_{\X^{-i}}) - F(X^i, m^0)  + F(y,m^0) - F(y, \mu^{N-1}_{\X^{-i}})\\ 
      &\le \sup_{x \in \R^d}\big(F(x, \mu^{N-1}_{\X^{-i}}) - F(x, m^0)\big) + \big(F(y, m^0) - F(y, \mu^{N-1}_{\X^{-i}}) \big)=:\varepsilon^N.
    \end{align*}
This shows that $\X^N=(X^1,\ldots, X^N)$ is an $\varepsilon^N-$approximate Nash equilibrium. Since $m^0\in \cP_2(\R^d)$, observe that $A:=\{\lim_{N\to\infty}\mu^N_{\X^N}= m^0,~\mbox{weakly}\}=\{\lim_{N\to\infty}\mu^N_{\X^N}= m^0,~\mbox{in}~\cW_2\}$. As $F$ is Lipschitz-continuous by Assumption \ref{Ass.MFG} (i), it follows that $\varepsilon^N\longrightarrow 0$ pointwise on $A$. 
\end{proof}

\section{Particle system: contractivity and propagation of chaos}
\label{sec:particle_system_contractivity_and_propagation_of_chaos}

In this section, we organize the proof of Theorem \ref{thm:contracti.Unin.prop.chaos} into separate steps.

\begin{proof}[Proofs of Theorem \ref{thm:contracti.Unin.prop.chaos}]
Let $X$ and $\bar X$ be the solutions of the McKean-Vlasov Langevin SDE \eqref{eq:McVLangevin2} starting from $X_0$ and $\bar X_0$. Denoting $\bar m_t: = \P\circ \bar X_t^{-1}$ and $\Delta X_t: = X_t - \bar X_t$, it follows from  It\^o's formula that
  \begin{align}
    \frac{d}{dt}\E[|\Delta X_t|^2] 
    &= 
    2\E\Big[\Delta X_t\!\cdot\!\big[\nabla_{\!x} F(X_t, m_t) 
                                                                               \!-\! \nabla_{\!x} F(\bar X_t, \bar m_t) 
                                                                        \big]\Big]
   \le
   - 2\ell_F |\Delta X_t|^2.
   \label{proof(i)}
  \end{align}
by the $\ell_F$-displacement monotonicity of $F$.  
Choosing $\bar X_0\sim m^\sigma$, it follows from the definition of the Wasserstein distance that $\cW_2^2(m_{X_t}, m^\sigma)\le\E[|\Delta X_t|^2] \le \E[|X_0 - \bar X_0|^2]e^{-2\ell_Ft}$, by the Gronwall inequality.

We next turn to the propagation of chaos statement (ii).
Let $(X^i)_{i\ge1}$ be solutions of Equation \eqref{eq:McVLangevin2} driven by Brownian motions $W^i$ with initial condition $X^i_0$. Denote $\widetilde\X_t := (X^1,\dots, X^N)$ and recall that the $X^i$'s are i.i.d. Denoting $\Delta X^{i,N}_t:= X^{i,N}_t - X^i_t$, we obtain by It\^o's formula:
  \begin{align}
    \frac{d}{dt}|\Delta X^{i,N}_t|^2 
    &= 
    \!-\! 2\Delta X^{i,N}_t\!\cdot\! 
                         \Big[\nabla_{\!x} F(X^{i,N}_t, \mu^{N-1}_{\X_t^{-i}}) 
                                  \!-\! \nabla_{\!x} F(X^i_t, m_{X_t}) \Big]
     \notag\\
        & = - 2\Delta X^{i,N}_t\cdot \Big[\nabla_{\!x} F(X^{i,N}_t, \mu^{N}_{\X_t}) - \nabla_{\!x} F(X^i_t, \mu^N_{\widetilde \X_t}) \Big]
        \notag\\
        &\hspace{5mm} -2\Delta X^{i,N}_t\!\cdot\!\Big[\nabla_{\!x} F(X^{i,N}_t, \mu^{N-1}_{\X_t^{-i}}) 
                                                                              - \nabla_{\!x} F(X^{i,N}_t, \mu^{N}_{\X_t}) 
                                                                     \Big] 
        \label{3terms.DM}\\
        &\hspace{5mm} -2\Delta X^{i,N}_t\cdot\Big[\nabla_{\!x} F(X^i_t, \mu^N_{\widetilde\X_t})  
                                                                             -  \nabla_{\!x} F(X^i_t, m_{X_t})\Big]. 
       \notag\
  \end{align}
Averaging out the first term on the right hand side of \eqref{3terms.DM}, we may rewrite the resulting identity using the random measure $\mathcal{L}_{(\mu^N_{\X_s},
                                                             \mu^N_{\tilde \X_s})}
                                    :=\frac1N\sum_{i=1}^N \delta_{(X^{i,N}_s,X^i_s)}$ as
  \begin{align}
  \notag
    -\frac1N&\sum_{i=1}^N\Delta X^{i,N}_t\cdot \big[\nabla_{\!x} F(X^{i,N}_t, \mu^{N}_{\X_t}) - \nabla_{\!x} F(X^i_t, \mu^N_{\widetilde \X_t})
    \big]
    \\ \notag
    &\le - 2{\color{black}\int\! (x_1\!-\!x_2)
                                                    \!\cdot\!
                                                    \big[\big(\nabla_xF(x_1, \mu^N_{\X_t}) 
                                                                  \!-\! \nabla_xF(x_2,\mu^N_{\tilde\X_t})
                                                           \big)
                                                      \big]
                                                    \mathcal{L}_{(\mu^N_{\X_t},
                                                             \mu^N_{\tilde \X_t})}(dx_1,dx_2)
                         }
 \\ \label{eq:apply.LL.proof}
       &\le -\frac{2\ell_F}{N}\sum_{i=1}^N|\Delta X^{i,N}_t|^2,
  \end{align}
by the $\ell_F$-displacement monotonicity of $F$. 
As for the second term on the right hand side of \eqref{3terms.DM}, we use the definition of linear derivative to rewrite it as:
  \begin{align}
    -\Delta X^{i,N}_t\!\cdot &\big[(\nabla_{\!x} F(X^{i,N}_t, \mu^{N-1}_{\X_t^{-i}}) - \nabla_{\!x} F(X^{i,N}_t, \mu^{N}_{\X_t}) \big]
    \notag\\ 
    & \hspace{-10mm}
        =  -\Delta X^{i,N}_t\cdot\int_0^1\int_{\R^d}\delta_m\nabla_{\!x} F(X^{i,N}_t, \lambda \mu^{N-1}_{\X_t^{-i}} + (1- \lambda)\mu^N_{\X_t},\tilde x)( \mu^{N-1}_{\X_t^{-i}} - \mu^N_{\X_t})(d\tilde x)d\lambda
    \notag\\ 
    & \hspace{-10mm}
        = -\Delta X^{i,N}_t\!\cdot\!\int_0^1\!\!\bigg\{\frac{1}{N\!-\!1}\sum_{j=1,j\neq i}^{N-1}D_m F(X^{i,N}_t, \lambda \mu^{N-1}_{\X_t^{-i}} \!+\! (1\!-\! \lambda)\mu^N_{\X_s},X^{j,N}_t)
    \notag\\  
    &\hspace{20mm}  -\frac1N\sum_{j=1}^N D_m F(X^{i,N}_t, \lambda \mu^{N-1}_{\X_t^{-i}} \!+\! (1\!-\! \lambda)\mu^N_{\X_t},X^{j,N}_t)d\lambda  \bigg\} 
  \le 
  \frac{4C_F}{N}|\Delta X^{i,N}_t|,
  \label{eq:applyDm.bound}
  \end{align}
where the last inequality is due to the boundedness of $D_mF$ in Assumption \ref{Ass.MFG}. Finally, we estimate the last term in \eqref{3terms.DM} through the Cauchy-Schwarz inequality:
  $$
    \E\Big[\Delta X^{i,N}_t
               \!\cdot\! 
               \big( \nabla_{\!x} F(X^i_t, \mu^N_{\widetilde\X_t}) 
                       - \nabla_{\!x} F(X^i_t, m_{X_t})\big) \Big] 
    \le \big\|\Delta X^{i,N}_t\big\|_{\L^2}
          \big\|\nabla_{\!x} F(X^i_t, \mu^N_{\widetilde\X_t})-\nabla_{\!x} F(X^i_t, m_{X_t})\big\|_{\L^2},
  $$
and we estimate the right hand side by \citeauthor{jackson2023quantitative} \cite[Lemma 3.13]{jackson2023quantitative} to derive
\begin{equation}\label{eq:Lipsch.FG}
    \E\Big[\Delta X^{i,N}_t
               \!\cdot\! \big( \nabla_{\!x} F(X^i_t, \mu^N_{\widetilde\X_t}) 
                                                         \!-\!\nabla_{\!x} F(X^i_t, m_{X_t})\big) \Big] 
      \le \frac{\tilde C}{\sqrt{N}}\big(1\!+\!\E[|X^i_t|^2] \big)^{\frac12}
             \big\|\Delta X^{i,N}_t\big\|_{\L^2}
      \le \frac{C}{\sqrt{N}}
             \big\|\Delta X^{i,N}_t\big\|_{\L^2},
  \end{equation}
for some constants $\tilde C, C$ that do not depend on $N$ and $t$, and where the last inequality follows by Lemma \ref{lem:integrability}. Plugging \eqref{eq:apply.LL.proof}, \eqref{eq:applyDm.bound} and \eqref{eq:Lipsch.FG} in \eqref{3terms.DM} yields that the function $f(t):=\big(\frac1N\sum_{i=1}^N\E[|\Delta X^{i,N}_t|^2]\big)^{1/2}$ satisfies the following differential inequality: 
  \begin{align*}
  (f^2)'(t)
  &\le -2\ell_F f(t)\big[f(t)-\frac{C}{2\ell_F\sqrt{N}}\big].
  \end{align*}
As $f>0$, this can be written as $f'
 \le -\ell_F[f-\frac{C}{2\ell_F\sqrt{N}}]$, thus implying the estimate:
  \begin{equation}
  \label{eq:PoC.LL.last}
      \Big(\frac1N\sum_{i=1}^N\E[|\Delta X^{i,N}_t|^2]\Big)^{1/2} \le \Big(\frac1N\sum_{i=1}^N\E[|\Delta X^{i,N}_0|^2]\Big)^{1/2}e^{-\ell_Ft} +\Big(1 - e^{-\ell_Ft}\Big) \frac{C}{2\ell_F\sqrt{N}}.
  \end{equation}
  By \citet[Lemma 2.1]{jackson2023quantitative} and uniqueness of the SDE \eqref{eq:McVLangevin2}, it follows that $X^{i,N} = \cX(X^i_0,W^i_{[0,t]}, \mu^N[\X_t])$ and $X^i_t = \bar\cX(X^i_0, W^i_{[0,t]})$ for some measurable maps $\cX$ and $\bar\cX$.
  In particular, $\E[|\Delta X^{i,N}_t|^2] = \E[|\Delta X^{j,N}_t|^2]$ for all $i,j$, and \eqref{eq:PoC.LL.last} provides the required estimate. 
\end{proof}

The proofs of the Theorems \ref{thm:contracti.Unin.prop.chaos-iii} and \ref{thm:contracti.Unin.prop.chaos-iv} are deferred to Section \ref{sec:approximation_of_equilibria}. We conclude this section by proving contractivity with a weaker form of displacement monotonicity \eqref{eq:weakDM}.

\begin{proof}[Proof of Proposition \ref{thm:contracti.DMweak}]
      The proof is inspired by arguments used in \citeauthor{Liu-Toelle11} \cite{Liu-Toelle11} in a non-McKean-Vlasov case. Under the monotonicity condition \eqref{eq:weakDM}, it follows from the  dynamics of $\Delta X = X - \bar X$ in \eqref{proof(i)} that
\begin{eqnarray*}
    \frac{d}{dt}\E[|\Delta X_t|^2] 
    &=& 
    -2\E\big[\Delta X_t\!\cdot\! \big(\nabla_{\!x} F(X_t, m_{X_t})\!-\!\nabla_{\!x} F(\bar X_t, \bar m_t) 
      \big]
     \\
    &\le&
    - 2c_F\frac{\E[|\Delta X_t|^2]}{1 + \E[|X_t|] + \E[|\bar X_t|]}dt,
\end{eqnarray*}
by \eqref{eq:weakDM}. In particular, the function $t\longmapsto \E[|\Delta X_t|^2]$ is decreasing, and
  \begin{align*}
    \E[|\Delta X_t|^2] 
    &\le \E[|\Delta X_0|^2] - 2c_F\int_0^t\frac{\E[|\Delta X_s|^2]}{1 + \E[|X_s|] + \E[|\bar X_s|]}ds
    \\
    &\le \E[|\Delta X_0|^2] - 2c_F\E[|\Delta X_t|^2]\int_0^t\frac{ds}{1 + \E[|X_s|] + \E[|\bar X_s|]}.
  \end{align*}
Collecting the terms in $\E[|\Delta X_t|^2]$ and applying the Jensen inequality, this provides
  \begin{eqnarray*}
    \E[|\Delta X_t|^2] 
    &\le& \E[|\Delta X_0|^2]\bigg(1 + \int_0^t\frac{2c_F}{1 + \E[|X_s|] + \E[|\bar X_s|]}ds \bigg)^{-1}
    \\
    &\le& \E[|\Delta X_0|^2]\bigg(1 + \frac{2c_Ft}{1 + \frac1t\int_0^t\E[|X_s| + |\bar X_s|]ds}\bigg)^{-1}.
    \end{eqnarray*}
    By Lemma \ref{lem:integrability}, we have that $\frac1t\int_0^t\E[|X_s| + |\bar X_s|]ds\le C$ is bounded uniformly in $t$. Then
    \begin{eqnarray*}
    \E[|\Delta X_t|^2] 
    &\le&
    \E[|\Delta X_0|^2]\Big(1 + \frac{2c_Ft}{1 + C}\Big)^{-1}
    \;\le\; \E[|\Delta X_0|^2]\Big(\frac{2c_Ft}{1 + C}\Big)^{-1}
    \;=\; \E[|\Delta X_0|^2]\frac{1 + C}{2c_Ft}. 
    \end{eqnarray*}
\end{proof}

\section{Approximation of equilibria}
\label{sec:approximation_of_equilibria}
In preparation for the proofs of Theorems \ref{thm:contracti.Unin.prop.chaos-iii} and \ref{thm:contracti.Unin.prop.chaos-iv}, we derive an error approximation of the Wasserstein distance between the invariant measure $m^\sigma$ and the MFE $m^0$.
  
\begin{proposition}
\label{prop:small-sigma}
    Let Assumption \ref{Ass.MFG} be satisfied and assume that $F$ is $\ell_F$-displacement monotone with $\ell_F>0$.
    Let $m^\sigma$ be the invariant measure of the McKean-Vlasov SDE \ref{eq:McVLangevin2} and $m^0$ its limit as $\sigma\searrow 0$ .
    Then it holds
    \begin{equation}
    \label{eq:bound.sigma.first}
        \sup_{\sigma>0}\,
        \frac{1}{\sigma}\,\cW_2^2(m^\sigma,m^0) \le \frac{d}{\ell_F},
    \end{equation}
Moreover, there exists $C>0$ such that for every $1$-Lipschitz function $\varphi$, 
\begin{equation*}
  m^\sigma_{\otimes n}\Big[ \varphi - \E^{m^0_{\otimes n}}[\varphi]>r\Big] \le  \e^{-\frac{\ell_F}{2n\sigma}(r^2-Cnr\sqrt{\sigma})},
  ~\mbox{for all}~n\ge1, r>0~\mbox{and}~\sigma<\frac{r^2}{C^2n^2},
\end{equation*}
where $m^\sigma_{\otimes n}$ and $m^0_{\otimes n}$ are the $n$-fold products of $m^\sigma$ and $m^0$, respectively.
  \end{proposition}

\begin{proof}
{\bf 1.} We first prove \eqref{eq:bound.sigma.first}.\footnote{This proof of \eqref{eq:bound.sigma.first} was suggested to us by Vincent Hoffmann. We thank him for his insight.}
Let \(\pi \in \Pi(m^\sigma, m^0)\) be a coupling such that \(\mathcal{W}_2^2(m^\sigma,m^0) = \int_{\mathbb{R}^d} |x-x'|^2 \,\pi(dx,dx')\). Then using the displacement monotonicity, we get
\begin{align*}
\ell_F  \mathcal{W}_2^2(m^\sigma, m^0)
&  =
\ell_F  \int_{\mathbb{R}^d} |x-x'|^2 \,\pi(dx,dx')
\\
& \leq
\int_{\mathbb{R}^d} \big( \nabla_{\!x} F(x,m^\sigma)  - \nabla_x F(x',m^0)  \big) \cdot (x-x') \,\pi(dx,dx')
\\
&  =
\int_{\mathbb{R}^d} \!\!\!\nabla_{\!x} F(x,m^\sigma) \!\cdot\! (x\!-\!x') \pi(dx,dx')
- \!\!\int_{\mathbb{R}^d} \!\!\!\nabla_{\!x} F(x',m^0) \!\cdot\! (x\!-\!x') \pi(dx,dx')
=: \mathrm{(I)} + \mathrm{(II)}.
\end{align*}
Let us begin by estimating \(\mathrm{(II)}\). Using that \(m^0\) is a MFE and that \(F\) is differentiable yields
\[
\mathrm{supp}(m^0)
\subseteq
\argmin F(\cdot, m^0)
\subseteq
\big\{ x' \in \mathbb{R}^d \,\big|\, \nabla_x F(x',m^0) = 0 \big\}.
\]
This implies \(\nabla_x F(x',m^0) = 0\) for \(m^0\)-a.e.\@ \(x' \in \mathbb{R}^d\), and thus also for \(\pi\)-a.e.\@ \((x,x') \in \mathbb{R}^d \times \mathbb{R}^d\), so that
\[
\mathrm{(II)}=\int_{\mathbb{R}^d} \nabla_x F(x',m^0) \cdot (x-x') \,\pi(dx,dx')
=
0.
\]
For \(\mathrm{(I)}\), recall that \(m^\sigma\) is given by $m^\sigma(dx) = \frac{1}{Z^\sigma} e^{-V(x)} \,dx$, with \(V(x) := \frac{F(x,m^\sigma)}{\sigma}\).
Since \(m^\sigma\) is absolutely continuous with respect to the Lebesgue measure, Brenier's theorem implies that there exists a convex function \(\phi: \mathbb{R}^d \to \mathbb{R}\) such that \(\pi(dx,dx') = \delta_{\nabla \phi(x)}(dx') \,m^\sigma(dx)\), see \citet[Theorem 2.12]{villani2003topics}.
This allows us to calculate
\begin{align}
\label{EqFirstTermSum}
\mathrm{(I)}
&=
\int_{\mathbb{R}^d} \nabla_x F(x,m^\sigma) \cdot \big(x - \nabla \phi(x)\big) \,m^\sigma(dx) \nonumber
\\
&=
\frac{\sigma}{Z^\sigma} \int_{\mathbb{R}^d} \nabla V(x) \cdot \big(x - \nabla \phi(x) \big) \, e^{-V(x)} \,dx \nonumber
\\
&=
- \frac{\sigma}{Z^\sigma} \int_{\mathbb{R}^d} \nabla \big(e^{-V(x)}\big) \cdot x \,dx + \frac{\sigma}{Z^\sigma} \int_{\mathbb{R}^d} \nabla \big(e^{-V(x)}\big) \cdot \nabla \phi(x) \,dx.
\end{align}
Integrating by parts formula, we can write the first term on the right-hand side of \eqref{EqFirstTermSum} as
\[
- \frac{\sigma}{Z^\sigma} \int_{\mathbb{R}^d} \nabla \big( e^{-V(x)} \big) \cdot x \,dx
=
\frac{\sigma d}{Z^\sigma} \int_{\mathbb{R}^d} e^{-V(x)} \,dx
=
\sigma d.
\]
Indeed, for \(R > 0\), let \(\chi_R \in C_c^\infty(\mathbb{R}^d)\) be a nonnegative cutoff function such that \({\chi_R}_{|_{B_R}} = 1\), \({\chi_R}_{|_{\mathbb{R}^d \setminus B_{2R}} } = 0\), and \(|\nabla \chi_R| \leq \frac{C}{R}\) with \(C \geq 0\) independent of \(R\). Then \(\chi_R(x) e^{-V(x)} \in C^1_c(\mathbb{R}^d)\) so that the integration by parts formula yields
\begin{equation}
\label{EqEqualityIBP}
\int_{\mathbb{R}^d} \nabla \big( \chi_R(x) e^{-V(x)} \big) \cdot x \,dx
=
- \int_{\mathbb{R}^d} \chi_R(x) e^{-V(x)} \mathrm{div}(x) \,dx
=
- d \int_{\mathbb{R}^d} \chi_R(x) e^{-V(x)} \,dx.
\end{equation}
Moreover, \(\nabla (\chi_R(x) e^{-V(x)}) = \chi_R(x) \nabla(e^{-V(x)}) + e^{-V(x)} \nabla \chi_R(x)\), and we estimate
\[
\int_{\mathbb{R}^d} e^{-V(x)} \nabla \chi_R(x) \cdot x \,dx
\leq
\frac{C}{R} \int_{B_{2R} \setminus B_R} |x| e^{-V(x)} \,dx
\leq
\frac{C Z^\sigma}{R} \int_{\mathbb{R}^d} |x| \,m^\sigma(dx).
\]
So we arrive at the desired equality upon letting \(R \to \infty\) in \eqref{EqEqualityIBP}.
For the second term of \eqref{EqFirstTermSum}, recall that the weak derivative \(\nabla \phi\) of the convex function \(\phi\) has locally bounded variation and that
\[
\int_{\mathbb{R}^d} \nabla \phi(x) \cdot \nabla \psi(x)dx
=
- \int_{\mathbb{R}^d} \psi(x) \,\Delta \phi(dx), \quad \forall \psi \in C^1_c(\mathbb{R}^d),
\]
where \(\Delta \phi\) is a (nonnegative) Radon measure, see \citet[Theorem 5.1 and Theorem 6.8]{Evans-Gariepy15}.
Thus, since \(\chi_R(x) e^{-V(x)} \geq 0\), we can estimate
\[
\frac{\sigma}{Z^\sigma} \int_{\mathbb{R}^d} \nabla \big( \chi_R(x) e^{-V(x)} \big) \cdot \nabla \phi(x) \,dx
\leq
0,
\]
or equivalently,
\[
\frac{\sigma}{Z^\sigma} \int_{\mathbb{R}^d} \chi_R(x) \nabla \big(e^{-V(x)} \big) \cdot \nabla \phi(x) \,dx
\leq
- \frac{\sigma}{Z^\sigma} \int_{\mathbb{R}^d} e^{-V(x)} \nabla \chi_R(x) \cdot \nabla \phi(x) \,dx.
\]
Using that \(\nabla \phi \in L^2(m^\sigma)\), we get
\[
\int_{\mathbb{R}^d} \big| e^{-V(x)} \nabla \chi_R(x) \cdot \nabla \phi(x) \big| \,dx
\leq
\frac{C Z^\sigma}{R} \int_{\mathbb{R}^d} \big| \nabla \phi(x) \big| \,m^\sigma(dx)
\xrightarrow{R \to \infty} 0,
\]
and we conclude that \(\mathrm{(II)} \leq \sigma d\).
This proves the desired result.

\medskip

\noindent {\bf 2.} We next turn to the derivation of the concentration inequality.
Since the function $F(\cdot,m^\sigma)$ is twice continuously differentiable and $\delta :=\frac{\ell_F}{\sigma}>0$, it follows from \citet[Corollary 7.2]{Goz-Leon10} that the probability measure $m^\sigma$ satisfies the quadratic transport inequality $\cW^2_2(m^\sigma, \mu) \le \frac{2}{\delta}H(\mu|m^\sigma)$ for all $\mu \in \cP_2(\R^d)$.
By standard concentration inequality, see e.g. \citet[Theorem 3.1]{Bob-Got99} or \citet[Theorems 5.1 $\&$ 5.2]{delarue2020master} we have for every $1$-Lipschitz function $\varphi:(\R^d)^n\longrightarrow\R$ and all $r>0$, $n\ge 1$, that $m^\sigma_{\otimes n}(\varphi - \int\varphi dm^\sigma_{\otimes n}>r)\le e^{-r^2\delta/2n}$. Then
\begin{align*}
  r - \int_{\R^d}\varphi d(m^\sigma_{\otimes n}-m^0_{\otimes n}) 
  &\ge r - \cW_1(m^\sigma_{\otimes n}, m^0_{\otimes n})
  \ge r - Cn\sqrt{\sigma}
  ~~\mbox{for all}~~r>0,
  \end{align*}
  where the last estimate follows from \eqref{eq:bound.sigma.first}.
  Thus, for $\sigma< r^2/C^2n^2$ we have $r - \int\varphi d(m^\sigma_{\otimes n}-m^0_{\otimes n})>0$ and therefore
\begin{align*}
  m^\sigma_{\otimes n}
  \Big[ \varphi - \int\varphi dm^0_{\otimes n}>r\Big] 
 &= m^\sigma_{\otimes n}
 \Big[\varphi - \int\varphi dm^\sigma_{\otimes n}
        > r - \int\varphi d(m^\sigma_{\otimes n}-m^0_{\otimes n})
 \Big]\\
 &\le e^{-\frac{\delta}{2n}(r-\int\varphi d(m^0_{\otimes n}-m^\sigma_{\otimes n})
                                              )^2 }\\ 
  &\le e^{-\frac{\delta}{2n}(r^2 -2r \int\varphi d(m^0_{\otimes n}-m^\sigma_{\otimes n}))}.
\end{align*}
Notice that $\int\varphi d(m^\sigma_{\otimes n}-m^0_{\otimes n})\le n\cW_1(m^\sigma,m^0)\le Cn\sqrt{\sigma}(\ell_F+ \sigma\ell_U)^{-1/2}$ by \eqref{eq:bound.sigma.first}.
Hence,
\begin{align*}
  m^\sigma_{\otimes n}
  \Big[ \varphi - \int\varphi dm^0_{\otimes n}>r\Big] 
 &\le e^{-\frac{\delta}{2n}(r^2 -Cnr\sqrt{\sigma}(\ell_F+\sigma\ell_U)^{-1/2})},
\end{align*}
which is the required result, given the definition of $\delta$. 
\end{proof}

\begin{proof}[Proof of Theorem \ref{thm:contracti.Unin.prop.chaos-iv}] Recall that because $F$ is Lasry and Lions monotone and convex, it is displacement monotone.
  With this observation, the estimate in {\rm (i)} follows by triangle inequality and then application of Proposition \ref{prop:small-sigma} and Theorem \ref{thm:contracti.Unin.prop.chaos}.{\rm (i)}.
  The proof of {\rm (ii)} also follows by the triangle inequality and Theorem \ref{thm:contracti.Unin.prop.chaos}.{\rm (ii)} and Proposition \ref{prop:small-sigma}.
\end{proof}

\begin{proof}[Proof of Theorem \ref{thm:contracti.Unin.prop.chaos-iii} {\rm (i)}] Let $N$ be fixed.
Recall the vector $ \X^{\sigma,N}$ defined in Proposition \ref{Prop:Conv.N-Nash.to.avera}, denote by $ m_t$ the law of the solution $X$ of the McKean-Vlasov equation \eqref{eq:McVLangevin2}, and let $ \X^\sigma_t := (X^1_t,\dots,X^N_t)$ where $X^i$ solves \eqref{eq:McVLangevin2} with Brownian motion $W^i$.
 By the triangle inequality, we have for all $t,\sigma> 0$:
\begin{align}
\notag
    \cW_p^p(\mu^N_\bx, m^0) 
    &\le C\bigg( \cW_p^p\Big(\mu^N_\bx, \frac1t\int_0^t\mu^{N}_{ \boldsymbol u_s}ds\Big) 
            + \cW_p^p\Big(\frac1t\int_0^t\mu^{N}_{ \boldsymbol u_s}ds,\frac1t\int_0^t \mu^{N}_{ \X^{\sigma,N}_s}ds\Big)\\ \notag
      &\quad \qquad    + \cW_p^p\Big(\frac1t\int_0^t \mu^{N}_{ \X^{\sigma,N}_s}ds, \frac1t\int_0^t \mu^{N}_{ \X_s^\sigma}ds\Big) + \cW^p_p\Big(\frac1t\int_0^t \mu^{N}_{ \X^\sigma_s}ds,\frac1t\int_0^tm^\sigma_sds\Big)\\ 
      &\quad \qquad  + \cW_p^p\Big(\frac1t\int_0^tm^\sigma_sds,m^\sigma\Big) + \cW_p^p(m^\sigma,m^0) 
              \bigg)
    =: C\sum_{k=1}^6J_k
     \label{eq:def.J5}
\end{align}
where, to be precise, we use the convention $\frac1t\int_0^t\eta_sds(A) := \frac1t\int_0^t\eta_s(A)ds$ for any measurable family of Borel probability measures $(\eta_s)_{s\in [0,t]}$ on $\R^d$. 
In particular, for any sufficiently integrable function $f$ it holds $\int_{\R^d}fd(\frac1t\int_0^t\eta_sds) = \frac1t\int_0^t\int_{\R^d}fd\eta_sds$, see e.g. \citeauthor[Proposition 3.2]{Cruc-Tangp24} \cite[Proposition 3.2]{Cruc-Tangp24}.
For $J_1$, notice that since $F$ is Lasry-Lions monotone and convex, it is therefore displacement monotone. 
As a consequence, the functions $F_i(x^1,\dots, x^N) := F(x^i, \mu^{N-1}_{\bx^{-i}})$ are strictly displacement monotone (in the sense of Definition \ref{def:nonsym}) and convex in $x_i$.
Thus,
\begin{align*}
  J_1:= \cW_p^p\Big(\mu^N_\bx, \frac1t\int_0^t\mu^{N}_{ \boldsymbol u_s}ds\Big)  \le \frac1N\sum_{i=1}^N\Big|x^i - \frac1t\int_0^tu^{i,N}_sds\Big| \xrightarrow [t\to \infty]{} 0
\end{align*}
where the limit follows by \cite[Proposition 3.1]{Awi-Hynd-Mawi23} and uniqueness of the Nash equilibrim.

For the second term, we have
\begin{align*}
  \E[J_2]= \E\bigg[\cW_p^p\Big(\frac1t\int_0^t\mu^{N}_{ \boldsymbol u_s}ds,\frac1t\int_0^t \mu^{N}_{ \X^{\sigma,N}_s}ds\Big)\bigg] &\le \frac{1}{Nt}\sum_{i=1}^N\int_0^t\E[|u^{i,N}_s - X^{i,N}_s|^p]ds.
\end{align*}
Then, it follows from 
Equation \eqref{eq:Deltaf} that 
$$
\lim_{\sigma\to 0}\lim_{t\to\infty} \E[J_2]= 0.
$$

We next deduce from Theorem \ref{thm:contracti.Unin.prop.chaos} {\rm (ii)} that
\begin{align*}
  \E[J_3] 
  &= \E\bigg[ \cW_p^p\Big(\frac1t\int_0^t \mu^{N}_{ \X^{\sigma,N}_s}ds, \frac1t\int_0^t \mu^{N}_{ \X_s^\sigma}ds\Big) \bigg ]
  \le \frac1t\frac{1}{N}\sum_{i=1}^N\int_0^t\E[|X^{i,N}_s \!-\! X^i_s|^p]ds \\ 
  & \le \frac1t\frac{1}{N}\sum_{i=1}^N \int_0^t\E[|X^{i,N}_s \!-\! X^i_s|^2]^{p/2}ds
  \le \frac{C}{N^{\frac{p}{2}}},
\end{align*}
for some constant $C$ independent of $t$. Finally, as $(X^1,\dots, X^N)$ are i.i.d. copies of $X$, we have 
\begin{align*}
  \E[J_4] 
  =
  \E\bigg[\cW^p_p\Big(\frac1t\int_0^t \mu^{N}_{ \X_s^\sigma}ds,\frac1t\int_0^tm^\sigma_sds\Big)\bigg] 
  &\le \frac1t\int_0^t\!\!\E\big[\cW_p^p(\mu^N_{\X_s^\sigma},m^\sigma_{s})\big]ds
  \le C\delta_{N,p} \E[|X_t|^2]^{p/2}\le C\delta_{N,p},
\end{align*}
where the last inequalities follows from 
\citeauthor{fournier2015rate} \cite[Theorem 1]{fournier2015rate}  and Lemma \ref{lem:integrability}.
By Theorem \ref{thm:contracti.Unin.prop.chaos}, we have
\begin{align*}
  J_5 = \cW_p^p\Big(\frac1t\int_0^tm^\sigma_sds,m^\sigma\Big) \le \frac{C}{2t\ell_F}\Big( 1 - e^{-2\ell_F t} \Big) \xrightarrow [t\to \infty]{} 0
\end{align*}
and by Proposition \ref{prop:small-sigma} $J_6 =\cW_p^p(m^\sigma,m^0)\le \cW^p_2(m^\sigma,m^0) \le C\sigma^{p/2}/\ell_F^{p/2}$.
Thus, taking the limit in $N,t$ and $\sigma$ in \eqref{eq:def.J5} concludes the proof.
\end{proof}

\begin{proof}[Proof of Theorem \ref{thm:contracti.Unin.prop.chaos-iii}{\rm (ii)}] 
Let $\X^\sigma = (X^{1,\sigma},\dots, X^{N,\sigma})$ where $X^{i,\sigma}$ are i.i.d. random variables with law $m^\sigma$.
By the triangle inequality, we have
\begin{align*}
  \P\big[\cW_1(\mu^N_{\X_t}, \mu_{\bx}^N)>r\big] 
  & \le \P\big[ \cW_1(\mu^N_{\X_t}, \mu_{\X^\sigma_t}^N) \ge r/5\big] 
  + \P\big[ \cW_1(\mu^N_{\X_t^\sigma}, m^\sigma) - \E[\cW_1(\mu^N_{\X_t^\sigma}, m^\sigma)] \ge r/5\big]\\ 
  &\quad + \P[\E[\cW_1(\mu^N_{\X^\sigma_t}, m^\sigma)]\!\ge\! r/5] 
              +\P\big[ \cW_1(m^\sigma\!, m^0)\!>\!r/5\big]  
              + \P\big[\cW_1(m^0\!, \mu_{\bx}^N) \!\ge\! r/5\big].
\end{align*}
Since the function $\Phi:\by\mapsto \cW_1(\mu^N_{\by}, m^\sigma)$ is $1$-Lipschitz, it follows from Proposition \ref{prop:small-sigma} that
\begin{equation*}
  \P\big[ \cW_1(\mu^N_{\X^\sigma_t}, m^\sigma) - \E[\cW_1(\mu^N_{\X^\sigma_t}, m^\sigma)] \ge r/5\big] 
  \le \e^{-\frac{\ell_F}{10N\sigma}(r^2/5-CNr\sqrt{\sigma})}  
  =: \varepsilon.
\end{equation*}
On the other hand, since by Theorem \ref{thm:contracti.Unin.prop.chaos} we have $\E[\cW_1(\mu^N_{\X_t}, \mu_{\X^\sigma_t}^N)] \le \frac1N\sum_{i=1}^N\E[|X^{i,N}_t - X^{i,\sigma}_t|]\le \frac{C}{N} + Ce^{-Ct}$ for some $C>0$, it follows that there is $N_0$ and $t_0$ large enough such that for all $N\ge N_0$, $t\ge t_0$, we have $\P\big[ \cW_1(\mu^N_{\X_t}, \mu_{\X^\sigma_t}^N) \ge r/5\big]\le \varepsilon$. 
Similarly, by the law of large numbers, Proposition \ref{prop:small-sigma} (with $n=1$ therein) and the first part of the proof that we can choose $N$ large enough and $\sigma$ small enough that $\P\big[\E[\cW_1(\mu^N_{\X^\sigma}, m^\sigma)]\ge r/5\big] = \P\big[ \cW_1(m^\sigma, m^0)>r/5\big]  = \P\big[\cW_1(m^0, \mu_{\bx}^N) \ge r/5\big] = 0$.
Thus,
\begin{equation*}
    \P\big[\cW_1(\mu^N_{\X_t}, \mu_{\bx}^N)>r\big] \le 2\varepsilon,
\end{equation*}
which concludes the proof.
\end{proof}


\end{document}